\documentclass[a4paper,10pt,reqno]{amsart}

\usepackage{enumerate}
\usepackage{amssymb, amsmath}

\usepackage{mathrsfs}
\usepackage{amscd}
\usepackage[active]{srcltx}
\usepackage{verbatim}
\usepackage{color}
\definecolor{darkred}{rgb}{0.9,0.1,0.1}

\usepackage[colorlinks,linkcolor={black},citecolor={black},urlcolor={red}]{hyperref}

\theoremstyle{plain}
\newtheorem{theorem}{Theorem}[section]
\newtheorem{corollary}[theorem]{Corollary}
\newtheorem{lemma}[theorem]{Lemma}
\newtheorem{proposition}[theorem]{Proposition}

\newtheorem{definition}[theorem]{Definition}

\theoremstyle{remark}
\newtheorem{remark}[theorem]{Remark}
\newtheorem{example}[theorem]{Example}

\numberwithin{equation}{section}

\def\N{{\mathbb N}}

\def\R{{\mathbb R}}

\newcommand{\E}{{\mathbb E}}
\renewcommand{\P}{{\mathbb P}}
\newcommand{\F}{{\mathscr F}}

\renewcommand{\d}{\delta}
\newcommand{\eps}{\varepsilon}
\newcommand{\om}{\omega}
\newcommand{\Om}{\Omega}

\DeclareMathOperator{\Dom}{\mathsf{D}}

\newcommand{\ip}[1]{\langle {#1}\rangle}

\newcommand{\dd}{\;\mathrm{d}}

\newcommand{\beq}{\begin{equation}}
\newcommand{\eeq}{\end{equation}}
\newcommand{\bal}{\begin{aligned}}
\newcommand{\eal}{\end{aligned}}
\newcommand{\ben}{\begin{enumerate}}
\newcommand{\beni} {\begin{enumerate}[(i)]}
\newcommand{\een}{\end{enumerate}}
\newcommand{\bit}{\begin{itemize}}
\newcommand{\eit}{\end{itemize}}
\newcommand{\beqw}{\begin{equation*}}
\newcommand{\eeqw}{\end{equation*}}
\newcommand{\bex}{\begin{example}}
\newcommand{\eex}{\end{example}}
\newcommand{\bre}{\begin{example}}
\newcommand{\ere}{\end{example}}
\newcommand{\bma}{\begin{bmatrix}}
\newcommand{\ema}{\end{bmatrix}}

\newcommand{\one}{{{\bf 1}}}

\renewcommand{\hat}{\widehat}
\newcommand{\ot}{\otimes}

\newcommand{\hN}{\tilde{N}}
\newcommand{\Y}{\mathscr{E}}
\newcommand{\YY}{E}

\newcommand{\cL}{\mathscr{L}}

\newcommand{\cC}{\mathscr{C}(\Om)}
\newcommand{\cCX}{\mathscr{C}(\Om;X)}
\newcommand{\cCXs}{\mathscr{C}(\Om;X^*)}

\newcommand{\kk}{\mathbf{k}}

\newcommand{\n}{\Vert}

\newcommand{\jj}{\mathbf{j}}
\newcommand{\ee}{\mathbf{e}}
\newcommand{\rr}{\mathbf{r}}
\newcommand{\BB}{\mathbf{B}}

\newcommand{\JJ}{\mathbf{J}}
\newcommand{\KK}{\mathbf{K}}

\newcommand{\dD}{\mathbb{D}}

\newcommand{\lb}{\langle}
\newcommand{\rb}{\rangle}

\newcommand{\cP}{\mathscr{P}}

\newcommand{\cD}{\mathscr{D}}
\newcommand{\cI}{\mathscr{I}}
\newcommand{\cJ}{\mathscr{J}}
\newcommand{\cF}{\mathscr{F}}
\newcommand{\cG}{\mathscr{G}}
\newcommand{\cB}{\mathscr{B}}
\newcommand{\cS}{\mathscr{S}}
\newcommand{\ti}{\times}
\newcommand{\si}{\sigma}
\newcommand{\bF}{\mathbb{F}}

\newcommand{\xspace}{\hbox{\kern-2.5pt}}

\newcommand{\bP}{\mathbb{P}}
\newcommand{\cE}{\mathscr{E}}
\newcommand{\Ex}{\mathbb{E}}
\newcommand{\oto}{}
\newcommand{\otw}{}
\newcommand{\Si}{\Sigma}
\newcommand{\cA}{\mathscr{A}}
\newcommand{\al}{\alpha}

\allowdisplaybreaks

\begin{document}

\title[Poisson stochastic integration in Banach spaces]{Poisson stochastic
integration in Banach spaces}

\author{Sjoerd Dirksen}
\address{Hausdorff Center for Mathematics\\
University of Bonn\\
Endenicher Allee 60\\
53115 Bonn\\
Germany} \email{sjoerd.dirksen@hcm.uni-bonn.de}

\author{Jan Maas}
\address{
Institute for Applied Mathematics\\
University of Bonn\\
Endenicher Allee 60\\
53115 Bonn\\
Germany}
\email{maas@uni-bonn.de}

\author{Jan van Neerven}
\address{
Delft Institute of Applied Mathematics\\
Delft University of Technology
\\P.O. Box 5031\\
2600 GA Delft\\
The Netherlands}
\email{J.M.A.M.vanNeerven@tudelft.nl}

\thanks{The first and third named authors were supported by VICI subsidy
639.033.604 of the Netherlands Organisation for Scientific Research (NWO). The
first named author was also supported by the Hausdorff Center for Mathematics.
}



 \begin{abstract}
We prove new upper and lower bounds for Banach space-valued stochastic integrals with respect to a compensated Poisson random measure. Our estimates apply to Banach spaces with non-trivial martingale (co)type and extend various results in the literature.
We also develop a Malliavin framework to interpret Poisson stochastic integrals as vector-valued Skorohod integrals, and prove a Clark-Ocone representation formula.
 \end{abstract}

\date\today

\maketitle

\section{Introduction}

This paper investigates upper and lower bounds for the $L^p$-norms of stochastic integrals of the form
\begin{equation}
\label{eqn:ItoIntro}
\int_{\R_+\times J} \phi\dd \tilde N,
\end{equation}
where $\tilde N$ is a compensated Poisson random measure with jump space $J$ and $\phi$ is a simple, adapted process.
Such bounds translate into necessary and sufficient conditions, respectively, for It\^{o} $L^p$-stochastic integrability.
It is well known that if $\phi$ takes values in a Hilbert space $H$, then by a straightforward generalization of the
Wiener-It\^{o} isometry one has
\begin{align*}
\E\Big\|\int_{\R_+\ti J} \phi  \ \dd\tilde{N}\Big\|_{H}^2 = \E\|\phi\|_{L^2(\R_+\times J;H)}^2.
\end{align*}
It is however surprisingly more difficult to find bounds for the $p$-th moments of stochastic integrals for $p\neq 2$
and for processes $\phi$ taking values in more general Banach spaces.\par
In a recent paper \cite{Dir12}, the first-named author obtained sharp upper and lower bounds for the $L^p$-norms
or in other words, an It\^o isomorphism, for stochastic integrals of the form (\ref{eqn:ItoIntro}) in the case
that $\phi$ takes values in $L^q(S)$ with $1<q<\infty$. These estimates take six fundamentally different forms
depending on the relative position of the parameters $p$ and $q$ with respect to $2$, see
Theorem~\ref{thm:summarySILqPoisson} for a precise statement. This is in sharp contrast to the
situation for stochastic integrals with respect to Wiener noise where, essentially as a consequence of Kahane's inequalities, the space of stochastically integrable processes can be described in terms of a single family of
square function norms. In fact, an It\^{o} isomorphism for Gaussian stochastic integrals in the much wider class of
UMD Banach spaces was obtained in \cite{vNVW07}. The result in \cite{Dir12} indicates, however, that the situation is
more involved for Poisson stochastic integrals and it remains an open problem to find sharp bounds for such integrals
with a Banach space-valued integrand.\par
The aim of the first part of this article is to obtain one-sided extensions of the estimates in \cite{Dir12} in more
general Banach spaces. We present upper bounds for the $L^p$-norm of the stochastic integral for Banach spaces
with non-trivial martingale type and lower bounds for Banach spaces with finite martingale cotype. The main upper bounds in
Theorem~\ref{thm:typesBanach} state that if $X$ has martingale type $s\in (1,2]$, then
\begin{align*}
\E\sup_{t>0}\Big\|\int_{(0,t]\ti J} \phi  \
\dd\tilde{N}\Big\|_{X}^p \lesssim_{p,s,X} \E\|\phi\|_{L^s(\R_+\times J;X)\cap L^p(\R_+\times J;X)}^p
\end{align*}
if $1<s\leq p<\infty$, while
\begin{align*}
\E\sup_{t>0}\Big\|\int_{(0,t]\ti J} \phi  \
\dd\tilde{N}\Big\|_{X}^p \lesssim_{p,s,X} \E\|\phi\|_{L^s(\R_+\times J;X)+L^p(\R_+\times J;X)}^p.
\end{align*} if $1<p\leq s$.
If $X$ has martingale cotype $s\in [2,\infty)$, then the `dual' inequalities
hold, see Theorem~\ref{thm:cotypesBanach}. Moreover, in case $X$ is a
Hilbert space, the martingale type and cotype inequalities combine into a
two-sided inequality that completely characterizes
the class of $L^p$-stochastically integrable processes. These statements
extend and complement various partial results in the literature
\cite{BrH09,Hau11,MPR10,MaR10,Zhu11}, see also the discussion after
Theorem~\ref{thm:typesBanach}.

As an application, we present in Theorem~\ref{thm:maxStochConv} some estimates for stochastic convolutions using a
standard dilation argument. In the setting of Hilbert spaces, these lead to
sharp maximal inequalities for stochastic convolutions with a semigroup of
contractions.

In general, the estimates in Theorems~\ref{thm:typesBanach} and
\ref{thm:cotypesBanach} do not lead to an It\^{o} isomorphism if $X$ is not a
Hilbert space. Indeed, the aforementioned result in \cite{Dir12} shows that
these estimates are already suboptimal for $L^q$-spaces with $q\neq 2$. For UMD
Banach spaces, however, we can still formulate an `abstract' It\^{o}-type
isomorphism. Using a standard decoupling argument we obtain, for any
$1<p<\infty$, a two-sided estimate of the form
\begin{equation}
\label{eqn:ItoUMDIntro}
\Big(\E \sup_{t>0}\Big\n \int_{(0,t]\times J} \phi \dd \tilde N
\Big\n_X^p\Big)^\frac1p \eqsim_{p,X} \n \phi\n_{L^p(\Om;\nu_p(\R_+\times
J;X))},
\end{equation}
where $\nu_p(E;X)$ is the completion of the space of simple functions $f:E\to X$
with respect to the {\em Poisson $p$-norm} introduced in Section \ref{sec:UMD};
the implied constants depend only on $p$ and $X$. Although the Poisson $p$-norm
is in general still a stochastic object, we can calculate it in terms of
deterministic quantities in case $X$ is a Hilbert space or an $L^q$-space.

The isomorphism \eqref{eqn:ItoUMDIntro} serves as a basis for the development of
a vector-valued Poisson Skorohod integral. In Section~\ref{sec:Malliavin} we
define a Malliavin derivative associated with a Poisson random measure in the
Banach space-valued setting following the Gaussian approach of \cite{MvNCO}. By
deriving an integration by parts formula we show that the Malliavin derivative
is a closable operator $D$ with respect to the Poisson $p$-norms. Assuming that
the Banach space $X$ is UMD, the adjoint operator $D^*$ is shown to extend the
It\^{o} stochastic integral with respect to the compensated Poisson random
measure (Theorem~\ref{thm:Skorokhod}). We conclude by proving a Clark-Ocone
representation formula in Theorem~\ref{thm:ClarkOcone}.
Our results extend similar results obtained by many authors in a scalar-valued
setting. References to this extensive literature are given in Sections
\ref{sec:Malliavin} and \ref{sec:CO}. To the best of our knowledge, the Banach
space-valued case has not been considered before.

\section{The Poisson stochastic integral}\label{sec:Poisson-SI}

We start by recalling the definition of a Poisson random measure and its
associated compensated Poisson random measure.
 Let $(\Om,\cF,\bP)$ be a
probability space and let $(E,\cE)$ be a measurable space.
We write $ \overline \N = \N \cup\{\infty\}$.

An \emph{integer-valued random measure} is a mapping
$N: \Omega\times\cE\to \overline \N$
with the following properties:
\begin{enumerate}
\renewcommand{\labelenumi}{\rm(\roman{enumi})}
\item For all $B\in\cE$ the mapping $N(B): \omega\mapsto N(\om,B)$ is measurable;
\item For all $\om\in\Om$, the mapping $B\mapsto N(\om,B)$ is a measure;
\end{enumerate}
The measure $\mu(B):= \E N(B)$ is called the \emph{intensity measure} of $N$.

\begin{definition}
An integer-valued random measure $N : \Omega\times\cE\to \overline \N $
with $\sigma$-finite intensity measure $\mu$ is called a {\em Poisson random measure} if the
following conditions are satisfied:
\begin{enumerate}
\renewcommand{\labelenumi}{\rm(\roman{enumi})}
\addtocounter{enumi}{2}
\item For all $B \in \cE$ the random variable $N(B)$ is
Poisson distributed with parameter $\mu(B)$;
\item For all pairwise disjoint sets $B_1,\ldots,B_n$ in $\cE$ the random variables
$N(B_1)$, $\ldots$, $N(B_n)$ are independent.
\end{enumerate}
\end{definition}
If $\mu(B) = \infty$ it is understood that $N(B) = \infty$ almost surely.
For the basic properties of Poisson random measures we refer to \cite[Chapter 6]{Cin}.

For $B \in \cE$ with  $\mu(B)<\infty$ we write
$$\tilde{N}(B) := N(B) - \mu(B).$$
It is customary to call $\tilde N$  the \emph{compensated Poisson random measure} associated with $N$
(even though it is not a random measure in the above sense, as it is defined on the sets of finite $\mu$-measure only).

Let $(J,\cJ,\nu)$ be a $\si$-finite measure space and let $N$ be a Poisson
random measure on $(\R_+\ti J,\cB(\R_+)\ti \cJ)$ with intensity measure $dt\ti \nu$. Throughout this
section we let
$\mathbb{F} = (\cF_t)_{t>0}$ be the filtration generated by the random
variables $\{N((s,u]\ti A) \ : \ 0\leq s<u\leq t, A \in \cJ\}$.

\begin{definition}
\label{def:simplePoisson-SI} Let $X$ be a Banach space. A process $\phi :\Om\ti \R_+\ti J\rightarrow X$ is a
\emph{simple, adapted $X$-valued process} if there are a finite partition
$\pi=\{0=t_1<t_2<\ldots<t_{l+1}<\infty\}$, random variables $F_{ijk} \in L^\infty(\Om,\cF_{t_i})$,
disjoint sets $A_1,\ldots,A_m$ in $\cJ$ satisfying
$\nu(A_j)<\infty$, and vectors
$x_{ijk} \in X$
for $i=1,\ldots,l$, $j=1,\ldots,m$, and $k=1,\ldots,n$ such that
\begin{equation}
\label{eqn:simpleBanach} \phi  = \sum_{i=1}^l\sum_{j=1}^m\sum_{k=1}^{n}F_{ijk}
\one_{(t_i,t_{i+1}]} \one_{A_j} x_{ijk}.
\end{equation}
Let $t>0$ and $B \in \cJ$. We define the \emph{(compensated) Poisson
stochastic integral of $\phi $} on $(0,t]\ti B$ with respect to $\tilde{N}$
by
\begin{equation*}
I_{t,B}(\phi ) = \int_{(0,t]\ti B} \phi  \dd\tilde{N} =
\sum_{i=1}^l\sum_{j=1}^m\sum_{k=1}^n F_{ijk} \tilde{N}((t_{i}\wedge
t,t_{i+1}\wedge t]\times (A_j\cap B)) x_{ijk},
\end{equation*}
where $s\wedge t = \min\{s,t\}$. For brevity, we write $I(\phi)$ to denote $I_{\infty,J}(\phi)$.
\end{definition}
\begin{definition}
\label{def:LpStochInt} Let $1\leq p<\infty$. A process $\phi :\Om\ti \R_+\ti
J\rightarrow X$ is said to be \emph{$L^p$-stochastically integrable} with
respect to
$\tilde{N}$ if there exists a sequence $(\phi _n)_{n\geq 1}$ of simple, adapted
$X$-valued processes such that
\begin{enumerate}[\rm(1)]
\item $\phi _n\rightarrow \phi $ almost everywhere;
\item For any $B \in \cJ$ and $t>0$ the sequence $\int_{(0,t]\ti B}
\phi _n\dd\tilde{N}$ converges in $L^p(\Om;X)$.
\end{enumerate}
In this case we define
\begin{equation*}
\int_{(0,t]\ti B} \phi  \dd\tilde{N} =
\lim_{n\rightarrow\infty} \int_{(0,t]\ti B} \phi _n \dd\tilde{N},
\end{equation*}
where the limit is taken in $L^p(\Om;X)$.
\end{definition}
A weaker notion of stochastic integrability, where $L^p$-convergence in the
second condition in Definition~\ref{def:LpStochInt} is replaced by convergence
in probability, has been studied extensively by Rosi\'{n}ski \cite{Ros84}.

\subsection{Conditions for \texorpdfstring{$L^p$}{Lp}-stochastic integrability}

To give substance to the class of $L^p$-stochastically integrable processes, we
study two different inequalities below. First, we search for so-called
Bichteler-Jacod inequalities
\begin{equation}
\label{eqn:B-J}
\Big(\E\Big\|\int_{\R_+\times J} \phi  \dd\tilde{N}\Big\|_{X}^p\Big)^{\frac{1}{p}}
\lesssim_{p,X} \|\phi \|_{\cI_{p,X}},
\end{equation}
where $\cI_{p,X}$ is a suitable Banach space. An inequality of this form implies
that every element in the closure of all simple, adapted processes in
$\cI_{p,X}$ is $L^p$-stochastically integrable. We also consider the reverse
inequality
\begin{equation}
\label{eqn:B-JRev}
\Big(\E\Big\|\int_{\R_+\times J} \phi  \dd\tilde{N}\Big\|_{X}^p\Big)^{\frac{1}{p}}
\gtrsim_{p,X} \|\phi \|_{\cI_{p,X}},
\end{equation}
which provides a necessary condition for stochastic integrability. Our main
results, Theorems~\ref{thm:typesBanach} and \ref{thm:cotypesBanach}, provide
inequalities of the forms \eqref{eqn:B-J} and \eqref{eqn:B-JRev}, respectively,
based on the martingale type and cotype of the Banach space $X$. Let us first
recall the definition of these notions.
\begin{definition}
If $X$ is a Banach space and $1\leq s\leq 2$, then $X$ is said to have
\emph{martingale type $s$} if every finite martingale difference sequence
$(d_i)$ in
$L^s(\Om;X)$ satisfies
$$\Big(\Ex\Big\|\sum_i d_i\Big\|_X^s\Big)^{\frac{1}{s}} \lesssim_{s,X}
\Big(\sum_i\Ex\|d_i\|_X^s\Big)^{\frac{1}{s}}.$$
On the other hand, if $2\leq r<\infty$, then $X$ is said to have
\emph{martingale cotype $r$} if every finite martingale difference sequence
$(d_i)$ in
$L^r(\Om;X)$ satisfies
$$\Big(\sum_i\Ex\|d_i\|_X^r\Big)^{\frac{1}{r}} \lesssim_{r,X}
\Big(\Ex\Big\|\sum_i d_i\Big\|_X^r\Big)^{\frac{1}{r}}.$$
\end{definition}
The notions of martingale type and cotype have been well studied by Pisier in
\cite{Pis75}, who established that $X$ has martingale type $s$ if and only if it
is $s$-uniformly smooth and has martingale cotype $r$ if and only if it is
$r$-uniformly convex. We shall use the following extrapolation principle,
observed in \cite[Remark 3.3]{Pis75}.
\begin{theorem}
\label{thm:PisierMtype} If $1<s\leq 2$, then $X$ has martingale type $s$ if and
only if for some (then any) $1\leq p<\infty$ every finite martingale difference
sequence
in $L^p(\Om;X)$ satisfies
\begin{equation}
\label{eqn:PisierMtypeEst} \Big(\Ex\Big\|\sum_i d_i\Big\|_X^p\Big)^{\frac{1}{p}}
\lesssim_{p,s,X}
\Big(\E\Big(\sum_i\|d_i\|_X^s\Big)^{\frac{p}{s}}\Big)^{\frac{1}{p}}.
\end{equation}
Moreover, if $2\leq r<\infty$, then $X$ has martingale cotype $r$ if and only if
for some (then any) $1\leq p<\infty$,
$$\Big(\E\Big(\sum_i\|d_i\|_X^r\Big)^{\frac{p}{r}}\Big)^{\frac{1}{p}}
\lesssim_{p,r,X} \Big(\Ex\Big\|\sum_i d_i\Big\|_X^p\Big)^{\frac{1}{p}}.$$
\end{theorem}
To obtain Bichteler-Jacod inequalities we need an estimate for the right hand
side of \eqref{eqn:PisierMtypeEst}, which is the content of
Lemma~\ref{lem:littleS} below. In the proof of this lemma, we shall use
the Burkholder-Rosenthal inequalities (see \cite[inequality (21.5)]{Bur73}): if
$2\leq p<\infty$, then for any scalar-valued martingale difference sequence
$(d_i)$ in $L^p(\Om)$,
\begin{equation}
\label{eqn:ClassicalBR}
\Big(\E\Big|\sum_{i=1}^n d_i\Big|^p\Big)^{\frac{1}{p}} \lesssim_p
\max\Big\{\Big(\sum_{i=1}^n\E|d_i|^p\Big)^{\frac{1}{p}},
\Big(\E\Big(\sum_{i=1}^n
\E_{i-1}|d_i|^2\Big)^{\frac{p}{2}}\Big)^{\frac{1}{p}}\Big\}.
\end{equation}
We also need the dual version of Doob's maximal inequality (see
\eqref{eqn:Stein} below for a more general statement). Suppose that $1\leq
p<\infty$. If $(f_i)$ is a sequence of non-negative random variables, then
\begin{equation}
\label{eqn:dualDoob}
\E\Big(\sum_{i\geq 1} \E_i(f_i)\Big)^p \lesssim_p \E\Big(\sum_{i\geq 1} f_i\Big)^p.
\end{equation}
\begin{lemma}
\label{lem:littleS} Let $X$ be a Banach space. Let $(\cF_i)_{i\geq 1}$ be a
filtration and let $(\E_i)_{i\geq 1}$ denote the sequence of associated
conditional expectations, with the convention that $\E_0 = \E$. Let $(\xi_i)$ be a finite sequence in $L^p(\Om;X)$
which is adapted to the given filtration. If $0<s\leq p<\infty$, then
\begin{equation}
\label{eqn:littleSLessp} \Big(\Ex\Big(\sum_i
\|\xi_i\|_X^s\Big)^{\frac{p}{s}}\Big)^{\frac{1}{p}} \eqsim_{p,s}
\max\Big\{\Big(\sum_i
\Ex\|\xi_i\|_X^p\Big)^{\frac{1}{p}}, \Big(\E\Big(\sum_i
\Ex_{i-1}\|\xi_i\|_X^s\Big)^{\frac{p}{s}}\Big)^{\frac{1}{p}}\Big\}.
\end{equation}
\end{lemma}
\begin{proof}
Since $s\leq p$ we have $l^s\hookrightarrow l^p$ contractively and so
$$\Big(\sum_i \Ex\|\xi_i\|_X^p\Big)^{\frac{1}{p}} \leq \Big(\Ex\Big(\sum_i
\|\xi_i\|_X^s\Big)^{\frac{p}{s}}\Big)^{\frac{1}{p}}.$$
Moreover, \eqref{eqn:dualDoob} implies
$$\Big(\E\Big(\sum_i
\Ex_{i-1}\|\xi_i\|_X^s\Big)^{\frac{p}{s}}\Big)^{\frac{1}{p}} \lesssim_{p,s}
\Big(\Ex\Big(\sum_i \|\xi_i\|_X^s\Big)^{\frac{p}{s}}\Big)^{\frac{1}{p}}.$$
This shows that $\gtrsim_{p,s}$ in \eqref{eqn:littleSLessp} holds. For the
proof of the reverse inequality, note that $\E\|\xi_i\|_X^s<\infty$ for all $i$.
By the triangle inequality,
$$\Big(\Ex\Big(\sum_i \|\xi_i\|_X^s\Big)^{\frac{p}{s}}\Big)^{\frac{s}{p}} \leq
\Big(\Ex\Big|\sum_i \|\xi_i\|_X^s -
\Ex_{i-1}\|\xi_i\|_X^s\Big|^{\frac{p}{s}}\Big)^{\frac{s}{p}}
 + \Big(\E\Big(\sum_i
\Ex_{i-1}\|\xi_i\|_X^s\Big)^{\frac{p}{s}}\Big)^{\frac{s}{p}}.$$
Notice that the sequence $(\|\xi_i\|_X^s - \Ex_{i-1}\|\xi_i\|_X^s)_{i\geq 1}$ is
a martingale difference sequence. Suppose first that $p\leq 2s$. In this case we
can apply the martingale type $\frac{p}{s}$ property of $\R$ to find
$$\Big(\Ex\Big|\sum_i \|\xi_i\|_X^s -
\Ex_{i-1}\|\xi_i\|_X^s\Big|^{\frac{p}{s}}\Big)^{\frac{s}{p}} \lesssim_{p,s}
\Big(\sum_i \Ex\Big|\|\xi_i\|_X^s -
\Ex_{i-1}\|\xi_i\|_X^s\Big|^{\frac{p}{s}}\Big)^{\frac{s}{p}}.$$
The result now follows by the triangle inequality and the conditional Jensen
inequality,
\begin{align}
\label{eqn:estpLess2s}
\Big(\sum_i \Ex\Big|\|\xi_i\|_X^s -
\Ex_{i-1}\|\xi_i\|_X^s\Big|^{\frac{p}{s}}\Big)^{\frac{s}{p}} & \leq \Big(\sum_i
\Ex\|\xi_i\|_X^p\Big)^{\frac{s}{p}} +
\Big(\sum_i\E(\Ex_{i-1}\|\xi_i\|_X^s)^{\frac{p}{s}}\Big)^{\frac{s}{p}} \nonumber
\\
& \leq 2 \Big(\sum_i \Ex\|\xi_i\|_X^p\Big)^{\frac{s}{p}}.
\end{align}
In case $p>2s$ we apply the Burkholder-Rosenthal inequalities
\eqref{eqn:ClassicalBR} to find
\begin{align*}
\Big(\Ex\Big|\sum_i \|\xi_i\|_X^s -
\Ex_{i-1}\|\xi_i\|_X^s\Big|^{\frac{p}{s}}\Big)^{\frac{s}{p}} & \lesssim_{p,s}
\max\Big\{\Big(\sum_i\Ex\Big|\|\xi_i\|_X^s -
\Ex_{i-1}\|\xi_i\|_X^s\Big|^{\frac{p}{s}}\Big)^{\frac{s}{p}}, \\
& \qquad \Big(\E\Big(\sum_i\Ex_{i-1}\Big(\|\xi_i\|_X^s -
\Ex_{i-1}\|\xi_i\|_X^s\Big)^2\Big)^{\frac{p}{2s}}\Big)^{\frac{s}{p}}\Big\}.
\end{align*}
We estimate the first term on the right hand side as in \eqref{eqn:estpLess2s}.
To the second term we apply the conditional triangle
inequality,
\begin{equation*}
\Big(\sum_i \E_{i-1}|f_i + g_i|^2\Big)^{\frac{1}{2}} \leq \Big(\sum_i
\E_{i-1}|f_i|^2\Big)^{\frac{1}{2}} + \Big(\sum_i
\E_{i-1}|g_i|^2\Big)^{\frac{1}{2}},
\end{equation*}
which is valid for all sequences of scalar-valued random variables $(f_i)_{i\geq
1}$ and $(g_i)_{i\geq 1}$, and find
\begin{align*}
& \Big(\E\Big(\sum_i \Ex_{i-1}\Big(\|\xi_i\|_X^s -
\Ex_{i-1}\|\xi_i\|_X^s\Big)^2\Big)^{\frac{p}{2s}}\Big)^{\frac{s}{p}} \\
& \qquad \leq  \Big(\E\Big(\sum_i
\Ex_{i-1}\|\xi_i\|_X^{2s}\Big)^{\frac{p}{2s}}\Big)^{\frac{s}{p}}
+ \Big(\E\Big(\sum_i
(\Ex_{i-1}\|\xi_i\|_X^s)^2\Big)^{\frac{p}{2s}}\Big)^{\frac{s}{p}} \\
& \qquad \leq \Big(\E\Big(\sum_i
\Ex_{i-1}\|\xi_i\|_X^{2s}\Big)^{\frac{p}{2s}}\Big)^{\frac{s}{p}}
 + \Big(\E\Big(\sum_i
\Ex_{i-1}\|\xi_i\|_X^s\Big)^{\frac{p}{s}}\Big)^{\frac{s}{p}}.
\end{align*}
Since $s\leq 2s\leq p$ we can find $0\leq \theta\leq 1$ such that $\frac{1}{2s}
= \frac{\theta}{s} + \frac{1-\theta}{p}$. By the conditional H\"{o}lder
inequality,
\begin{align*}
\Big(\sum_i \Ex_{i-1}\|\xi_i\|_X^{2s}\Big)^{\frac{1}{2s}} & \leq \Big(\sum_i
\Ex_{i-1}\|\xi_i\|_X^{s}\Big)^{\frac{\theta}{s}}\Big(\sum_i
\Ex_{i-1}\|\xi_i\|_X^{p}\Big)^{\frac{1-\theta}{p}} \\
& \leq \max\Big\{\Big(\sum_i
\Ex_{i-1}\|\xi_i\|_X^{s}\Big)^{\frac{1}{s}},\Big(\sum_i
\Ex_{i-1}\|\xi_i\|_X^{p}\Big)^{\frac{1}{p}}\Big\}.
\end{align*}
As a consequence,
$$\Big(\E\Big(\sum_i
\Ex_{i-1}\|\xi_i\|_X^{2s}\Big)^{\frac{p}{2s}}\Big)^{\frac{s}{p}} \leq
\max\Big\{\Big(\sum_i\Ex\|\xi_i\|^p\Big)^{\frac{1}{p}}, \Big(\E\Big(\sum_i
\Ex_{i-1}\|\xi_i\|_X^s\Big)^{\frac{p}{s}}\Big)^{\frac{1}{p}}\Big\}^s.$$
Collecting our estimates, we conclude that \eqref{eqn:littleSLessp} holds.
\end{proof}
In the next lemma we use the following inequality due to E.M. Stein (see \cite{Ste70},
Chapter IV, the proof of Theorem 8). Let $1<p<\infty$ and $1\leq s<\infty$. If $(f_i)_{i\geq 1}$
is a sequence of scalar-valued random variables, then
\begin{equation}
\label{eqn:Stein}
\Big\|\Big(\sum_i|\E_i(f_i)|^s\Big)^{\frac{1}{s}}\Big\|_{L^p(\Om)}
\lesssim_{p,s} \Big\|\Big(\sum_i|f_i|^s\Big)^{\frac{1}{s}}\Big\|_{L^p(\Om)}.
\end{equation}
\begin{lemma}
\label{lem:condIneqX} Let $X$ be a Banach space. Let $0=t_1<t_2<\ldots<t_{l+1}<\infty$, $F_{ijk} \in L^{\infty}(\Om)$, $x_{ijk} \in
X$,
and let $A_1,\ldots,A_m$ be disjoint sets in $\cJ$ satisfying $\nu(A_j)<\infty$
for $i=1,\ldots,l$, $j=1,\ldots,m$ and $k=1,\ldots,n$. Define
\[
\phi   :=  \sum_{i,j,k} F_{ijk} \oto \one_{(t_i,t_{i+1}]} \oto \one_{A_j} \oto
x_{ijk}
\]
and let $\tilde{\phi }$ be the associated simple adapted $X$-valued process given by
\[
\tilde{\phi }  :=  \sum_{i,j,k} \E(F_{ijk}|\cF_{t_i}) \oto \one_{(t_i,t_{i+1}]}
\oto \one_{A_j} \oto x_{ijk}.
\]
If $1<p<\infty$ and $1\leq s<\infty$, then
$$\|\tilde{\phi }\|_{L^p(\Om;L^s((0,t]\ti J;X))} \lesssim_{p,s}
\|\phi \|_{L^p(\Om;L^s((0,t]\ti J;X))}.$$
\end{lemma}
\begin{proof}
We may assume $t=t_{l+1}$. If we define $\cF_{i,j} = \cF_{t_i}$ for all $j$,
then $(\cF_{i,j})$ is a filtration with respect to the lexicographic ordering. By the
conditional Jensen inequality and Stein\textquoteright s inequality \eqref{eqn:Stein} we obtain
\begin{align*}
& \|\tilde{\phi }\|_{L^p(\Om;L^s((0,t]\ti J;X))} \\
& \qquad = \Big\|\Big(\sum_{i,j}(t_{i+1} -
t_i)\nu(A_j)\Big\|\sum_k\E(F_{ijk}|\cF_{t_i})x_{ijk}\Big\|_{X}^s\Big)^{\frac{1
}{s}}\Big\|_{L^p(\Om)} \\
& \qquad = \Big\|\Big(\sum_{i,j}(t_{i+1} - t_i)\nu(A_j)\Big\|\E\Big(\sum_k
F_{ijk}
x_{ijk}\Big|\cF_{i,j}\Big)\Big\|_{X}^s\Big)^{\frac{1}{s}}\Big\|_{L^p(\Om)} \\
& \qquad \leq \Big\|\Big(\sum_{i,j}(t_{i+1} -
t_i)\nu(A_j)\Big(\E\Big(\Big\|\sum_k F_{ijk}
x_{ijk}\Big\|_{X}\Big|\cF_{i,j}\Big)\Big)^s\Big)^{\frac{1}{s}}\Big\|_{L^p(\Om)
} \\
& \qquad \lesssim_{p,s} \Big\|\Big(\sum_{i,j}(t_{i+1} - t_i)\nu(A_j)\Big\|\sum_k
F_{ijk} x_{ijk}\Big\|_{X}^s\Big)^{\frac{1}{s}}\Big\|_{L^p(\Om)} \\
& \qquad = \|\phi \|_{L^p(\Om;L^s((0,t]\ti J;X))}.
\end{align*}
\end{proof}
We will use the following elementary observation (see \cite[Lemma 3.4]{Dir12}).
\begin{lemma}
\label{lem:PoissonMoments} Let $N$ be a Poisson distributed random variable with
parameter $0\leq \lambda\leq 1$. Then for every $1\leq p<\infty$ there exist
constants
$b_p, c_p>0$ such that
\begin{equation*}
b_p \lambda \leq \E|N-\lambda|^p \leq c_p \lambda.
\end{equation*}
\end{lemma}
\begin{remark}
\label{rem:assumptionSimpleProcess} By refining the partition $\pi$ in
Definition~\ref{def:simplePoisson-SI} if necessary, we can and will always assume
that
$(t_{i+1}-t_i)\nu(A_j)\leq 1$ for all $i=1,\ldots,l$, $j=1,\ldots,m$. This will
allow us to apply Lemma~\ref{lem:PoissonMoments} to the compensated Poisson
random
variables $\tilde{N}((t_{i}\wedge t,t_{i+1}\wedge t]\times (A_j\cap B))$.
\end{remark}
Finally, we shall use the following reverse of the dual Doob inequality in
\eqref{eqn:dualDoob}. The authors learned this result from \cite[Theorem 7.1]{JuX03}, where this result is even obtained for non-commutative random variables. We
give a simple proof for the commutative case following \cite{JuX03}, which
yields an improved constant.
\begin{lemma}[Reverse dual Doob inequality]
\label{lem:dualDoobRev} Fix $0<p\leq 1$. Let $\bF$ be a filtration and let $(\E_i)_{i\geq 1}$
be the associated sequence of conditional expectations.
If $(f_i)_{i\geq 1}$ is a sequence of non-negative random variables in $L^1(\Om)$, then
$$\Big(\E\Big|\sum_{i\geq 1} f_i\Big|^p\Big)^{\frac{1}{p}} \leq p^{-1} \Big(\E\Big|\sum_{i\geq 1}
\E_i(f_i)\Big|^p\Big)^{\frac{1}{p}}.$$
\end{lemma}
\begin{proof}
Fix $\eps>0$. Define two sequences $A_n = \sum_{i=1}^n f_i$ and $B_n = \eps +
\sum_{i=1}^n \E_i(f_i)$. Set $A=A_{\infty}$, $B=B_{\infty}$ and $B_0 = \eps$. Let
us first observe that by H\"{o}lder's inequality,
$$\E(A^p) = \E(A^pB^{-p(1-p)}B^{p(1-p)}) \leq
(\E(AB^{-(1-p)}))^p(\E(B^p))^{1-p}.$$
We will show that
$$\E(AB^{-(1-p)})\leq p^{-1}\E(B^p-\eps^p).$$
Indeed,
\begin{align*}
\E(AB^{-(1-p)}) & \leq \sum_{i\geq 1}\E(f_iB_i^{-(1-p)}) \\
& = \sum_{i\geq 1}\E(\E_i(f_i)B_i^{-(1-p)}) = \sum_{i\geq 1}\E((B_i-B_{i-1})B_i^{-(1-p)}).
\end{align*}
By the concavity of the map $x\mapsto x^p$, any $0 < a \leq b$ satisfy
$$(b-a)b^{p-1} \leq p^{-1}(b^p - a^p).$$
Therefore,
$$\E(AB^{-(1-p)}) \leq p^{-1}\sum_{i\geq 1}\E(B_i^p-B_{i-1}^p) \leq p^{-1}\E(B^p-\eps^p).$$
We conclude that 
$$\E A^p\leq (p^{-1}\E B^p - p^{-1}\eps^p)^p(\E B^p)^{1-p}.$$
The result now follows by letting $\eps\downarrow 0$.
\end{proof}
To formulate our Bichteler-Jacod inequalities we consider for $1\leq p,q<\infty$
the norms
$$\|\phi \|_{\cD_{q,X}^p} = \|\phi \|_{L^p(\Om;L^q(\R_+\times J;X))} =
\Big(\E\Big(\int_{\R_+\ti J} \|\phi \|_X^q \dd t\ti \!\dd\nu\Big)^{\frac{p}{q}}\Big)^{\frac{1}{p}}.$$
\begin{theorem}[Upper bounds and non-trivial martingale type]
\label{thm:typesBanach} Let $X$ be a martingale type $s$ Banach space. If $1<s\leq
p<\infty$, then for any simple, adapted $X$-valued process $\phi$ and any $B \in
\cJ$,
\begin{align}
\label{eqn:typesBanachMax}
\Big(\E\sup_{t>0}\Big\|\int_{(0,t]\ti B} \phi  \
\dd\tilde{N}\Big\|_{X}^p\Big)^{\frac{1}{p}} \lesssim_{p,s,X} \|\phi \one_{B}\|_{\cD_{s,X}^p\cap \cD_{p,X}^p}.
\end{align}
On the other hand, if $1<p\leq s$ then
\begin{align}
\label{eqn:typesBanach1ps}
\Big(\E\sup_{t>0}\Big\|\int_{(0,t]\ti B} \phi  \
\dd\tilde{N}\Big\|_{X}^p\Big)^{\frac{1}{p}} \lesssim_{p,s,X} \|\phi \one_{B}\|_{\cD_{s,X}^p+\cD_{p,X}^p}.
\end{align}
\end{theorem}
\begin{proof}
Clearly, the process
$$\Big(\Big\|\int_{(0,t]\ti B} \phi  \dd\tilde{N}\Big\|_{X}\Big)_{t>0}$$
is a submartingale and therefore, by Doob's maximal inequality,
\begin{equation}
\label{eqn:DoobMaxInt}
\Big(\E\sup_{t>0}\Big\|\int_{(0,t]\ti B} \phi
\dd\tilde{N}\Big\|_{X}^p\Big)^{\frac{1}{p}} \leq
\frac{p}{p-1}\Big(\E\Big\|\int_{\R_+\ti B} \phi
\dd\tilde{N}\Big\|_{X}^p\Big)^{\frac{1}{p}}.
\end{equation}
Let $\phi $ be as in \eqref{eqn:simpleBanach}, taking
Remark~\ref{rem:assumptionSimpleProcess} into account. Without loss of
generality, we may assume that $B=J$. We write $\tilde{N}_{ij} :=
\tilde{N}((t_i,t_{i+1}]\ti A_j)$ for brevity. The sub-algebras
\begin{equation}
\label{eqn:filtLexico}
\cF_{(i,j)} := \si\Big(\cF_{t_i}, \ \tilde{N}((t_i,t_{i+1}]\ti A_k), \
k=1,\ldots,j\Big) \ \ \ \ (i=1,\ldots,l, \ j=1,\ldots,m)
\end{equation}
form a filtration if we equip the pairs $(i,j)$ with the lexicographic ordering.
We shall use $(i,j)-1$ to denote the element preceding $(i,j)$ in this ordering.
If we define $y_{ij} = \sum_k
F_{ijk} \oto x_{ijk}$, then $(d_{ij}) := (\tilde{N}_{ij}y_{ij})$ is a
martingale difference sequence with respect to the filtration
$(\cF_{(i,j)})_{i,j}$ and
$$\int_{\R_+\ti J} \phi \dd\tilde{N} = \sum_{i,j} d_{ij}.$$
Suppose first that $s\leq p<\infty$. By Theorem~\ref{thm:PisierMtype} and
Lemma~\ref{lem:littleS} we find
\begin{align*}
\Big(\E\Big\|\sum_{i,j}d_{ij}\Big\|_X^p\Big)^{\frac{1}{p}} & \lesssim_{p,s,X}
\Big(\E\Big(\sum_{i,j}\|d_{ij}\|_X^s\Big)^{\frac{p}{s}}\Big)^{\frac{1}{p}} \\
& \lesssim_{p,s} \max\Big\{\Big(\sum_{i,j}
\E\|d_{ij}\|_X^p\Big)^{\frac{1}{p}}, \Big(\E\Big(\sum_{i,j}
\E_{(i,j)-1}\|d_{ij}\|_X^s\Big)^{\frac{p}{s}}\Big)^{\frac{1}{p}}\Big\}.
\end{align*}
By Lemma~\ref{lem:PoissonMoments},
\begin{align*}
\Big(\E\Big(\sum_{i,j}
\E_{(i,j)-1}\|d_{ij}\|_X^s\Big)^{\frac{p}{s}}\Big)^{\frac{1}{p}}
& = \Big(\E\Big(\sum_{i,j}
\E|\tilde{N}_{ij}|^s\|y_{ij}\|_X^s\Big)^{\frac{p}{s}}\Big)^{\frac{1}{p}}
\nonumber \\
& \eqsim_s \Big(\E\Big(\sum_{i,j}
(t_{i+1}-t_i)\nu(A_j)\|y_{ij}\|_X^s\Big)^{\frac{p}{s}}\Big)^{\frac{1}{p}}
\nonumber \\
& = \|\phi \|_{\cD_{s,X}^p}.
\end{align*}
Similarly,
$$\Big(\sum_{i,j} \E\|d_{ij}\|_X^p\Big)^{\frac{1}{p}} \eqsim_p
\|\phi \|_{\cD_{p,X}^p}.$$
Putting our estimates together we find
\begin{equation*}
\Big(\E\Big\|\int_{\R_+\ti J} \phi  \dd\tilde{N}\Big\|_{X}^p\Big)^{\frac{1}{p}}
\lesssim_{p,s,X} \max\Big\{\|\phi \|_{\cD_{s,X}^p}, \|\phi \|_{\cD_{p,X}^p}\Big\}.
\end{equation*}
The result in the case $s\leq p<\infty$ now follows from
\eqref{eqn:DoobMaxInt}.\par
Suppose next that $1<p<s$. Fix $\eps>0$. Since $L^{\infty}(\Om)\ot
L^{\infty}(\R_+) \ot L^{\infty}(J) \ot X$ is dense in both $\cD_{p,X}^p$ and
$\cD_{s,X}^p$, we can find $\phi _1,\phi _2$ in $L^{\infty}(\Om)\ot L^{\infty}(\R_+)
\ot L^{\infty}(J) \ot X$ such that $\phi  = \phi _1 + \phi _2$ and
$$\|\phi _1\|_{\cD_{p,X}^p} + \|\phi _2\|_{\cD_{s,X}^p} \leq \|\phi \|_{\cD_{p,X}^p + \cD_{s,X}^p} + \eps.$$
Let $\cA$ be the sub-$\si$-algebra of $\cB(\R_+)\ti\cJ$ generated by the
sets $(t_i,t_{i+1}]\ti A_j$, $i=1,\ldots,l$, $j=1,\ldots,m$. Then $\phi  = \E(\phi _1|\cA) + \E(\phi _2|\cA)$ and
$\E(\phi _1|\cA),\E(\phi _2|\cA)$ are of the form
\begin{equation*}
\E(\phi_{\al}|\cA) = \sum_{i,j,k} F_{ijk,\al} \oto \one_{(t_i,t_{i+1}]} \oto \one_{A_j}
\oto x_{ijk,\al} \qquad (\al = 1,2).
\end{equation*}
Let $\tilde{\phi }_1,\tilde{\phi }_2$ be the associated simple adapted processes
\begin{equation*}
\tilde{\phi}_{\al} = \sum_{i,j,k} \E(F_{ijk,\al}|\cF_{t_i}) \oto
\one_{(t_i,t_{i+1}]} \oto \one_{A_j} \oto x_{ijk,\al} \qquad (\al=1,2),
\end{equation*}
then $\phi =\tilde{\phi }_1 + \tilde{\phi }_2$. Therefore,
\begin{equation}
\label{eqn:intSplit}
\Big\|\int_{\R_+\ti J} \phi  \dd\tilde{N}\Big\|_{L^p(\Om;X)} \leq
\Big\|\int_{\R_+\ti J} \tilde{\phi }_1 \dd\tilde{N}\Big\|_{L^p(\Om;X)} +
\Big\|\int_{\R_+\ti J} \tilde{\phi }_2 \dd\tilde{N}\Big\|_{L^p(\Om;X)}.
\end{equation}
For $\al=1,2$ write $y_{ij,\al} = \sum_k F_{ijk,\al} \oto x_{ijk}$ and let
$(d_{ij,\al})_{i,j}$ denote the martingale difference sequence $(\tilde{N}_{ij}y_{ij,\al})_{i,j}$. We apply
Theorem~\ref{thm:PisierMtype}, Jensen's inequality and Lemma~\ref{lem:PoissonMoments} to the first
term on the right hand side of \eqref{eqn:intSplit} and find
\begin{align*}
\Big\|\int_{\R_+\ti J} \tilde{\phi }_1 \dd\tilde{N}\Big\|_{L^p(\Om;X)} &
\lesssim_{p,s,X}
\Big(\E\Big(\sum_{i,j}\|d_{ij,1}\|_X^s\Big)^{\frac{p}{s}}\Big)^{\frac{1}{p}} \\
& \leq \Big(\sum_{i,j}\E\|d_{ij,1}\|_X^p\Big)^{\frac{1}{p}} \\
& = \Big(\sum_{i,j}\E|\tilde{N}_{ij}|^p \E\|y_{ij,1}\|_X^p\Big)^{\frac{1}{p}}
\\
& \eqsim_{p} \Big(\sum_{i,j}(t_{i+1}-t_i)\nu(A_j)
\E\|y_{ij,1}\|_X^p\Big)^{\frac{1}{p}} = \|\tilde{\phi }_1\|_{\cD_{p,X}^p}.
\end{align*}
Since vector-valued conditional expectations are contractive,
$$\|\tilde{\phi }_1\|_{\cD_{p,X}^p} \leq \|\E(\phi _1|\cA)\|_{\cD_{p,X}^p} \leq
\|\phi _1\|_{\cD_{p,X}^p}.$$
To the second term on the right hand side of \eqref{eqn:intSplit} we apply
Theorem~\ref{thm:PisierMtype}, Lemma \ref{lem:dualDoobRev} and
Lemma \ref{lem:PoissonMoments} to find
\begin{align*}
\Big\|\int_{\R_+\ti J} \tilde{\phi }_2 \dd\tilde{N}\Big\|_{L^p(\Om;X)} &
\lesssim_{p,s,X}
\Big(\E\Big(\sum_{i,j}\|d_{ij,2}\|_X^s\Big)^{\frac{p}{s}}\Big)^{\frac{1}{p}} \\
& \lesssim_{p,s}
\Big(\E\Big(\sum_{i,j}\E_{(i,j)-1}\|d_{ij,2}\|_X^s\Big)^{\frac{p}{s}}\Big)^{\frac{1
}{p}} \\
& = \Big(\E\Big(\sum_{i,j}\E|\tilde{N}_{ij}|^s
\|y_{ij,2}\|_X^s\Big)^{\frac{p}{s}}\Big)^{\frac{1}{p}} \\
& \eqsim_{s} \Big(\E\Big(\sum_{i,j}(t_{i+1}-t_i)\nu(A_j)
\|y_{ij,2}\|_X^s\Big)^{\frac{p}{s}}\Big)^{\frac{1}{p}} =
\|\tilde{\phi }_2\|_{\cD_{s,X}^p}.
\end{align*}
By Lemma~\ref{lem:condIneqX},
$$\|\tilde{\phi }_2\|_{\cD_{s,X}^p} \lesssim_{p,s,X} \|\E(\phi _2|\cA)\|_{\cD_{s,X}^p}
\leq \|\phi _2\|_{\cD_{s,X}^p}.$$
It follows that
$$\Big\|\int_{\R_+\ti J} \phi  \dd\tilde{N}\Big\|_{L^p(\Om;X)} \lesssim_{p,s,X}
\|\phi _1\|_{\cD_{p,X}^p} + \|\phi _2\|_{\cD_{s,X}^p} \leq \|\phi \|_{\cD_{p,X}^p+\cD_{s,X}^p}
+ \eps.$$
Since $\eps>0$ was arbitrary, we conclude in view of \eqref{eqn:DoobMaxInt} that
\eqref{eqn:typesBanach1ps} holds.
\end{proof}
Theorem~\ref{thm:typesBanach} extends several known vector-valued
Bichteler-Jacod inequalities in the literature. An estimate for $X$ equal to a
Hilbert space $H$ and $2\leq p<\infty$ was obtained in \cite[Lemma 3.1]{MPR10} (see also \cite{Kun04} for an earlier version of this result for $H=\R^n$). The
estimate in \eqref{eqn:typesBanachMax} is slightly stronger in this case. In
fact, we will see in Corollary~\ref{cor:Hilbert} below that it cannot be
improved. In \cite[Lemma 4]{MaR10} a slightly weaker inequality than
\eqref{eqn:typesBanachMax} was obtained in the special case $X=L^s$, $p=s\geq
2$. Note however, that the estimate in \eqref{eqn:typesBanachMax} is still
suboptimal in this case, see Theorem~\ref{thm:summarySILqPoisson} below. In
\cite{Hau11} Hausenblas proved \eqref{eqn:typesBanachMax} in the special case
$p=s^n$ for some integer $n\ge 1$.
Finally, Theorem~\ref{thm:typesBanach} has been obtained independently by Zhu
using a different approach (see the work in progress \cite{Zhu11}).\par
We now consider lower bounds for stochastic integrals of the form
\eqref{eqn:B-JRev}, based on the martingale cotype of the space.

\begin{lemma}
\label{lem:normingIntersection}
Let $X$ be a reflexive Banach space, let $(\Om,\cF,\bP)$ be a probability space and let
$(S,\Sigma,\mu)$ be a $\si$-finite measure space. Suppose $1<p,p',s,s'<\infty$
satisfy $\frac{1}{p} + \frac{1}{p'} = 1$ and $\frac{1}{s} + \frac{1}{s'} = 1$.
Then
$$
(L^{p}(\Om;L^{p}(S;X))+ L^{p}(\Om;L^{s}(S;X)))^*=
L^{p'}(\Om;L^{p'}(S;X^*))\cap L^{p'}(\Om;L^{s'}(S;X^*))$$
isometrically. The corresponding
duality bracket is given by
\begin{equation*}
\langle f,g\rangle = \E\int_S \langle f(s),g(s)\rangle \ d\mu(s),
\end{equation*}
for any $f \in L^{p}(\Om;L^{p}(S;X))+ L^{p}(\Om;L^{s}(S;X))$ and $g \in
L^{p'}(\Om;L^{p'}(S;X^*))\cap L^{p'}(\Om;L^{s'}(S;X^*))$.
\end{lemma}
\begin{proof}
The reflexivity of $X$ implies that for all $1<q,r<\infty$ we have
$$ (L^q(\Omega;L^r(S;X)))^* =  L^{q'}(\Omega;L^{r'}(S;X^*)).$$
The lemma now follows from the general fact
from interpolation theory (see \cite[Theorem 8.III]{AroGag}) that if $(Y_0, Y_1)$ is
an interpolation couple with $Y_0\cap Y_1$ dense in both $Y_0$ and $Y_1$, then
$(Y_0^*,Y_1^*)$ is an interpolation couple and
$(Y_0+Y_1)^* = Y_0^*\cap Y_1^*$ isometrically.
\end{proof}

In order to apply this lemma we recall (see \cite{Pis75}) that any Banach space with non-trivial martingale type
or finite martingale cotype is uniformly convex and therefore reflexive. In fact we will need only that
$L^{p'}(\Om;L^{p'}(S;X^*))\cap L^{p'}(\Om;L^{s'}(S;X^*))$ is norming for
$L^{p}(\Om;L^{p}(S;X))+ L^{p}(\Om;L^{s}(S;X))$; this is true for arbitrary Banach spaces $X$ and can
be proved with elementary means.

\begin{theorem}[Lower bounds and finite martingale cotype]
\label{thm:cotypesBanach} Let $X$ be a Banach space with martingale cotype
$2\leq s<\infty$. If $s\leq p<\infty$ we have, for any simple, adapted
$X$-valued process $\phi $
and any $B \in \cJ$,
\begin{equation*}
\|\phi \one_{
B}\|_{\cD_{s,X}^p\cap \cD_{p,X}^p} \lesssim_{p,s,X}
\Big(\E\Big\|\int_{\R_+\ti B} \phi  \dd\tilde{N}\Big\|_X^p\Big)^{\frac{1}{p}}.
\end{equation*}
On the other hand, if $1<p<s$ then
\begin{equation}
\label{eqn:cotypesBanach1ps} \|\phi \one_{
B}\|_{\cD_{s,X}^p+\cD_{p,X}^p}
\lesssim_{p,s,X} \Big(\E\Big\|\int_{\R_+\ti B} \phi
\dd\tilde{N}\Big\|_X^p\Big)^{\frac{1}{p}}.
\end{equation}
\end{theorem}
\begin{proof}
Let $\phi $ be the simple adapted process given in \eqref{eqn:simpleBanach},
taking Remark~\ref{rem:assumptionSimpleProcess} into account. We may assume that $B=J$.
Suppose first that $s\leq p<\infty$. If we define $y_{ij} = \sum_k F_{ijk} \oto x_{ijk}$,
then $(d_{ij}) := (\tilde{N}_{ij}y_{ij})$ is a martingale
difference sequence with respect to the filtration defined in \eqref{eqn:filtLexico} and
$$\int_{\R_+\ti B} \phi  \dd\tilde{N} = \sum_{i,j} d_{ij}.$$
By Theorem~\ref{thm:PisierMtype}, Lemma~\ref{lem:littleS} and Lemma~\ref{lem:PoissonMoments} we find
\begin{align*}
\Big(\E\Big\|\sum_{i,j}d_{ij}\Big\|_X^p\Big)^{\frac{1}{p}} & \gtrsim_{p,s,X}
\Big(\E\Big(\sum_{i,j}\|d_{ij}\|_X^s\Big)^{\frac{p}{s}}\Big)^{\frac{1}{p}} \\
& \eqsim_{p,s,X} \max\Big\{\Big(\sum_{i,j} \E\|d_{ij}\|_X^p\Big)^{\frac{1}{p}}\!,
\Big(\E\Big(\sum_{i,j}
\E_{(i,j)-1}\|d_{ij}\|_X^s\Big)^{\frac{p}{s}}\Big)^{\frac{1}{p}}\Big\} \\
& = \max\Big\{\Big(\sum_{i,j}
\E|\tilde{N}_{ij}|^p\E\|y_{ij}\|_X^p\Big)^{\frac{1}{p}}\!\!,
\Big(\E\Big(\sum_{i,j}
\E|\tilde{N}_{ij}|^s\|y_{ij}\|_X^s\Big)^{\frac{p}{s}}\Big)^{\frac{1}{p}}\Big\}
\\
& \eqsim_{p,s} \max\Big\{\Big(\sum_{i,j}
(t_{i+1}-t_i)\nu(A_j)\E\|y_{ij}\|_X^p\Big)^{\frac{1}{p}}, \\
& \qquad \qquad \qquad \qquad \Big(\E\Big(\sum_{i,j}
(t_{i+1}-t_i)\nu(A_j)\|y_{ij}\|_X^s\Big)^{\frac{p}{s}}\Big)^{\frac{1}{p}}\Big\}
\\
& = \|\phi \|_{\cD_{s,X}^p\cap \cD_{p,X}^p}.
\end{align*}
We deduce the inequality in the case $1<p<s$ by duality from Theorem~\ref{thm:typesBanach}.
By Lemma \ref{lem:normingIntersection}
 $\cD_{p',X^*}^{p'}\cap \cD_{s',X^*}^{p'}$ is norming for $\cD_{s,X}^p+\cD_{p,X}^p$.
We let $\langle
\cdot,\cdot\rangle$ denote the associated duality bracket. Let $\psi $ be an element
of the algebraic tensor product $L^{\infty}(\Om)\ot L^{\infty}(\R_+)\ot
L^{\infty}(J)\ot X^*$. Let $\cA$ be the sub-$\si$-algebra of $\cB(\R_+)\ti\cJ$
generated by the sets $(t_i,t_{i+1}]\ti A_j$. Then $\E(\psi |\cA)$ is of the form
$$\E(\psi |\cA) = \sum_{l,m,n} G_{lmn}\oto \one_{(t_l,t_{l+1}]} \oto \one_{A_m} \oto
x_{lmn}^*,$$
where $G_{lmn} \in L^{\infty}(\Om)$. Let $\tilde{\psi }$ be the associated simple
adapted process defined by
$$\tilde{\psi} = \sum_{l,m,n} \tilde{G}_{lmn}\oto \one_{(t_l,t_{l+1}]} \oto
\one_{A_m} \oto x_{lmn}^*,$$
where $\tilde{G}_{lmn}=\E(G_{lmn}|\cF_{t_l})$. By conditioning,
\begin{align*}
\langle \phi ,\psi \rangle & = \langle \phi ,\E(\psi |\cA)\rangle \\
& = \sum_{i,j,k,n}\E(F_{ijk} G_{ijn}) \dd t\ti \!\dd\nu((t_i,t_{i+1}]\ti A_j) \langle
x_{ijk},x_{ijn}^*\rangle \\
& = \sum_{i,j,k,n}\E(F_{ijk} \tilde{G}_{ijn}) \dd t\ti \!\dd\nu((t_i,t_{i+1}]\ti A_j)
\langle x_{ijk},x_{ijn}^*\rangle
\end{align*}
Observe now that
$$
\E (F_{ijk}\tilde{G}_{lmn}\tilde{N}_{ij}\tilde{N}_{lm}) = \left\{
  \begin{array}{ll}
   \!\!\E(F_{ijk}\tilde{G}_{ijn})\dd t\!\ti\! \!\dd\nu((t_i,t_{i+1}]\ti A_j), & \mathrm{\!if \ }
i=l \ \& \  j=m; \\
   \!\! 0, & \mathrm{\!otherwise.}
  \end{array}
\right.
$$
Therefore,
\begin{align*}
\langle \phi ,\psi \rangle & = \sum_{i,j,k,l,m,n} \E(F_{ijk}
\tilde{G}_{lmn}\tilde{N}_{ij}\tilde{N}_{lm}) \langle
x_{ijk},x_{lmn}^*\rangle \\
& = \E\Big(\Big\langle \sum_{i,j,k} F_{ijk} \tilde{N}_{ij} x_{ijk},
\sum_{l,m,n}\tilde{G}_{lmn} \tilde{N}_{lm} x_{lmn}^*\Big\rangle\Big) \\
& \leq \Big\|\sum_{i,j,k} F_{ijk} \tilde{N}_{ij} x_{ijk}\Big\|_{L^p(\Om;X)}
\Big\|\sum_{l,m,n}\tilde{G}_{lmn} \tilde{N}_{lm}
x_{lmn}^*\Big\|_{L^{p'}(\Om;X^*)} \\
& = \Big\|\int_{\R_+\ti J} \phi  \dd\tilde{N}\Big\|_{L^p(\Om;X)} \Big\|\int_{\R_+\ti
J} \tilde{\psi } \dd\tilde{N}\Big\|_{L^{p'}(\Om;X^*)}
\end{align*}
Since $X$ has martingale cotype $1<p<s$, $X^*$ has martingale type $1<s'\leq 2$
and $s'\leq p'$. Therefore, we can subsequently apply
Theorem~\ref{thm:typesBanach} and Lemma~\ref{lem:condIneqX} to obtain
\begin{align*}
\Big\|\int_{\R_+\ti J} \tilde{\psi } \dd\tilde{N}\Big\|_{L^{p'}(\Om;X^*)} &
\lesssim_{p,s,X} \|\tilde{\psi }\|_{\cD_{p',X^*}^{p'}\cap \cD_{s',X^*}^{p'}} \\
& \lesssim_{p,s} \|\E(\psi |\cA)\|_{\cD_{p',X^*}^{p'}\cap \cD_{s',X^*}^{p'}} \leq
\|\psi \|_{\cD_{p',X^*}^{p'}\cap \cD_{s',X^*}^{p'}}.
\end{align*}
Summarizing, we find
$$\langle \phi ,\psi \rangle \lesssim_{p,s,X} \Big\|\int_{\R_+\ti J} \phi
\dd\tilde{N}\Big\|_{L^p(\Om;X)}\|\psi \|_{\cD_{p',X^*}^{p'}\cap \cD_{s',X^*}^{p'}}.$$
By taking the supremum over all $\psi $ as above, we conclude that
\eqref{eqn:cotypesBanach1ps} holds.
\end{proof}
We obtain two-sided estimates for the $L^p$-norm of the stochastic integral in the special
case where $X$ has both martingale type and cotype equal to $2$. By Kwapie\'{n}'s
theorem (see e.g.\ \cite[Theorem 7.4.1]{AlK06}), such a space is isomorphic to a Hilbert space.
\begin{corollary}
\label{cor:Hilbert} Let $H$ be a Hilbert space. If $2\leq p<\infty$, then for any simple, adapted $H$-valued process $\phi $ and any
$B\in\cJ$,
\begin{equation*}
\Big(\E\sup_{t>0}\Big\|\int_{(0,t]\ti B} \phi \dd\tilde{N}\Big\|_H^p\Big)^{\frac{1}{p}} \eqsim_p \|\phi \one_{B} \|_{\cD^p_{2,H}\cap\cD^p_{p,H}}.
\end{equation*}
On the other hand, if $1<p<2$ then
\begin{equation*}
\Big(\E\sup_{t>0}\Big\|\int_{(0,t]\ti B} \phi
\dd\tilde{N}\Big\|_H^p\Big)^{\frac{1}{p}} \eqsim_p
\|\phi \one_{
B}\|_{\cD^p_{2,H}+\cD^p_{p,H}}.
\end{equation*}
\end{corollary}
The estimates in Corollary~\ref{cor:Hilbert} characterize the class of
$L^p$-stochastically integrable Hilbert space-valued functions. Outside the
setting of Hilbert spaces, the Bichteler-Jacod inequalities in
Theorem~\ref{thm:typesBanach} do not lead to optimal (i.e., two-sided) bounds
for the stochastic integral. Indeed, this is already true if $X$ is an
$L^q$-space with $q\neq 2$. In this case the following optimal bounds were
recently established in \cite{Dir12}. Let $(S,\Si,\mu)$ be any measure space and
consider the square function norm
$$\|\phi \|_{\cS_q^p} = \|\phi \|_{L^p(\Omega;L^q(S;L^2(\R_+\times J)))} =
\Big(\E\Big\|\Big(\int_{\R_+\ti J} |\phi |^2 \dd t\ti
\!\dd\nu\Big)^{\frac{1}{2}}\Big\|_{L^q(S)}^p\Big)^{\frac{1}{p}}.$$
\begin{theorem}[Two-sided bounds for $X = L^q(S)$ \cite{Dir12}]
\label{thm:summarySILqPoisson} Let $1<p,q<\infty$. For any $B\in\cJ$ and any simple, adapted $L^q(S)$-valued process
$\phi $,
\begin{equation*}
\Big(\E \sup_{t>0} \Big\|\int_{(0,t]\ti
B} \phi  \dd\tilde{N}\Big\|_{L^q(S)}^p\Big)^{\frac{1}{p}} \eqsim_{p,q}
\|\phi \one_{
B}\|_{\cI_{p,L^q}},
\end{equation*}
where $\cI_{p,L^q}$ is given by
\begin{align*}
\cS_{q}^p \cap \cD_{q,L^q(S)}^p \cap \cD_{p,L^q(S)}^p & \ \ \mathrm{if} \ \ 2\leq q\leq p;\\
\cS_{q}^p \cap (\cD_{q,L^q(S)}^p + \cD_{p,L^q(S)}^p) & \ \ \mathrm{if} \ \ 2\leq p\leq q;\\
(\cS_{q}^p \cap \cD_{q,L^q(S)}^p) + \cD_{p,L^q(S)}^p & \ \ \mathrm{if} \ \ p\le 2\leq q;\\
(\cS_{q}^p + \cD_{q,L^q(S)}^p) \cap \cD_{p,L^q(S)}^p & \ \ \mathrm{if} \ \ q\le 2\leq p;\\
\cS_{q}^p + (\cD_{q,L^q(S)}^p \cap \cD_{p,L^q(S)}^p) & \ \ \mathrm{if} \ \ q\leq p\leq 2;\\
\cS_{q}^p + \cD_{q,L^q(S)}^p + \cD_{p,L^q(S)}^p & \ \ \mathrm{if} \ \ p\leq q\leq 2.
\end{align*}
\end{theorem}
The result in Theorem~\ref{thm:summarySILqPoisson} can be further extended to integrands taking values
in a non-commutative $L^q$-space; see Section 7 of \cite{Dir12} for further details.

\section{Maximal inequalities for Poisson stochastic convolutions}\label{sec:Maximal}

It is a well-known strategy to derive maximal $L^p$-inequalities for stochastic
convolutions from Bichteler-Jacod inequalities by using a dilation of the
semigroup.
This approach was first utilized in \cite{HaS01}. A result similar to Theorem \ref{thm:maxStochConv}
has been obtained independently in \cite{Zhu11} by different methods and further ramifications are worked out there.
Let us say that a strongly
continuous semigroup $S$ on a Banach space $X$ \emph{has an isometric dilation on a Banach space $Y$} if there
exists an isomorphic embedding $Q:X\rightarrow Y$, a bounded projection
$P:Y\rightarrow QX$ and a strongly continuous group of isometries
$(U(t))_{t\in \R}$ on $Y$ such that $QS(t) = PU(t)Q$ for all $t>0$.
\begin{example}
\label{exa:dilation} Let us list some examples of semigroups which admit an isometric dilation.
\begin{enumerate}
\renewcommand{\labelenumi}{(\roman{enumi})}
\item A $C_0$-semigroup of contractions on a Hilbert space admits an isometric
dilation on a (possibly different) Hilbert space (Sz.-Nagy's dilation theorem);
\item Let $S$ be any measure space. If $1<p<\infty$, then any $C_0$-semigroup of
positive contractions on $L^p(S)$ admits an isometric dilation on
$L^p(\tilde{S})$, where $\tilde{S}$ is a possibly different measure space (cf.\
Fendler's dilation theorem \cite[Theorem 1]{Fen97});
\item \label{enu:R-sec} If $X$ is UMD space and $A$ is an R-sectorial operator
of type $<\frac{\pi}{2}$ on $X$, then the bounded analytic semigroup generated by $-A$
allows an isometric dilation on the space $\gamma(L^2(\R),X)$ of all
$\gamma$-radonifying operators from $L^2(\R)$ to $X$
if and only if $A$ has a bounded $H^{\infty}$-calculus \cite{FrW06}.
\end{enumerate}
\end{example}
We shall use the following simple approximation lemma.
\begin{lemma}
\label{lem:app}
Fix $1\leq p,s<\infty$ and let $X$ be a Banach space. Suppose that $\phi :\Om\ti\R_+\ti J\rightarrow X$ is a measurable process such that
for any $t \in \R_+$ and $j \in J$ the map $\om\rightarrow \phi (\om,t,j)$ is $\cF_t$-measurable. If $\phi  \in \cD_{s,X}^p\cap \cD_{p,X}^p$, then there is a sequence of simple, adapted processes $\phi _n$ converging to $\phi $ in $\cD_{s,X}^p\cap \cD_{p,X}^p$.
\end{lemma}
\begin{proof}
Since $\nu$ is a $\si$-finite measure, it suffices to show that $\phi \one_{(0,T]\ti
B}$ can be approximated in $L^p(\Om;L^s((0,T]\ti B;X))\cap L^p(\Om;L^p((0,T]\ti
B;X))$ by simple, adapted processes, for any fixed $T>0$ and set $B \in \cJ$ of
finite measure. Since
$$L^p(\Om;L^s((0,T]\ti B;X))\cap L^p(\Om;L^p((0,T]\ti B;X))$$
can be isomorphically (with constants depending on $T$ and $\nu(B)$) identified
with
$$L^p(\Om;L^p((0,T]\ti B;X)) \qquad \mathrm{if \ } s\leq p$$
and with
$$L^s(\Om;L^s((0,T]\ti B;X)) \qquad \mathrm{if \ } p\leq s,$$
the statement follows from \cite[Theorem 4.2]{Rud04}.
\end{proof}
\begin{theorem}[Maximal inequality for stochastic convolutions]
\label{thm:maxStochConv}
Fix $1<s\leq 2$ and let $X$ be a martingale type $s$ Banach space. Let $S$ be a bounded $C_0$-semigroup on $X$ and suppose that $S$ has an isometric
dilation on a martingale type $s$ space $Y$. If $s\leq p<\infty$, then any simple, adapted $X$-valued process $\phi$ satisfies
\begin{align*}
& \Big(\E\sup_{t>0}\Big\|\int_{(0,t]\ti B} S(t-u)\phi (u)
\dd\tilde{N}\Big\|_X^p\Big)^{\frac{1}{p}} \lesssim_{p,s,X,S} \|\phi \one_{
B}\|_{\cD_{s,X}^p\cap \cD_{p,X}^p}.
\end{align*}
On the other hand, if $1\leq p<s$ then
\begin{align*}
& \Big(\E\sup_{t>0}\Big\|\int_{(0,t]\ti B} S(t-u)\phi (u)
\dd\tilde{N}\Big\|_X^p\Big)^{\frac{1}{p}} \lesssim_{p,s,X} \|\phi \one_{
B}\|_{\cD_{s,X}^p + \cD_{p,X}^p}.
\end{align*}
\end{theorem}
\begin{proof}
In view of the identity $\one_B S(t-u)\phi(u) = S(t-u)(\phi(u)\one_B)$ there is no
loss of generality in assuming that $B=J$.

Let $\phi $ be a simple, adapted process. Since $S$ is a bounded $C_0$-semigroup, the map
$(\om,u,j)\mapsto \one_{(0,t]}(u)S(t-u)(\phi (\om,u,j))$ is strongly measurable and in
$\cD_{r,X}^p$ for $r=p,s$. Moreover, for every fixed $u$ and $j$, $\om\mapsto
\one_{(0,t]}(u)S(t-u)(\phi (\om,u,j))$ is $\cF_u$-measurable and by
Lemma~\ref{lem:app} and Theorem~\ref{thm:typesBanach} it is $L^p$-stochastically integrable.
Since $S$ has an isometric dilation on a martingale type $s$ space
$Y$, we can apply Theorem~\ref{thm:typesBanach} to obtain
\begin{align*}
& \Big(\E\sup_{t>0}\Big\|\int_{(0,t]\ti J} S(t-u)\phi (u)
\dd\tilde{N}\Big\|_X^p\Big)^{\frac{1}{p}} \\
& \qquad \eqsim_{Q} \Big(\E\sup_{t>0}\Big\|\int_{(0,t]\ti J}
QS(t-u)\phi (u) \dd\tilde{N}\Big\|_{Y}^p\Big)^{\frac{1}{p}} \\
& \qquad = \Big(\E\sup_{t>0}\Big\|\int_{(0,t]\ti J} PU(t-u)Q\phi (u)
\dd\tilde{N}\Big\|_{Y}^p\Big)^{\frac{1}{p}} \\
& \qquad \lesssim_P \Big(\E\sup_{t>0}\Big\|\int_{(0,t]\ti J} U(-u)Q\phi (u)
\dd\tilde{N}\Big\|_{Y}^p\Big)^{\frac{1}{p}}.
\end{align*}
The integrand in the last line is $L^p$-stochastically integrable, so by
Theorem~\ref{thm:typesBanach},
\begin{align*}
& \Big(\E\sup_{t>0}\Big\|\int_{(0,t]\ti J} U(-u)Q\phi (u)
\dd\tilde{N}\Big\|_{Y}^p\Big)^{\frac{1}{p}} \\
& \qquad \lesssim_{p,s,Y} \|U(-u)Q\phi (u)\|_{\cD_{s,Y}^p\cap \cD_{p,Y}^p} \\
& \qquad \lesssim_Q \|\phi \|_{\cD_{s,X}^p\cap \cD_{p,X}^p}.
\end{align*}
In the case $1<p<s$ it suffices to consider a decomposition $\phi =\phi _1 + \phi _2$, where
$\phi _1,\phi _2$ are simple adapted processes (see the reduction argument in the proof
of Theorem~\ref{thm:typesBanach}) and show that
\begin{align*}
\Big(\E\sup_{t>0}\Big\|\int_{(0,t]\ti J} S(t-u)\phi _1(u)
\dd\tilde{N}\Big\|_X^p\Big)^{\frac{1}{p}} & \lesssim_{p,s,X,S}
\|\phi _1\|_{\cD_{s,X}^p} \\
\Big(\E\sup_{t>0}\Big\|\int_{(0,t]\ti J} S(t-u)\phi _2(u)
\dd\tilde{N}\Big\|_X^p\Big)^{\frac{1}{p}} & \lesssim_{p,s,X,S}
\|\phi _2\|_{\cD_{p,X}^p}.
\end{align*}
These inequalities follow using a dilation argument as above.
\end{proof}
By applying Theorem~\ref{thm:maxStochConv} to the semigroups in (i) and (ii) of Example~\ref{exa:dilation} we find improvements of the maximal inequalities obtained in \cite{MaR10} and \cite{MPR10}. Indeed, if $X$ is a Hilbert space (so $s=2$), then Theorem~\ref{thm:maxStochConv} improves upon \cite[Lemma 4]{MaR10}. Since in this case our maximal estimates are already optimal for the trivial semigroup $S(t)\equiv\mathrm{Id}_H$ (c.f.\ Corollary~\ref{cor:Hilbert}), they are the best possible. If $X=L^q(S)$ (so that $s=q$) and $q=p$, then we find a sharpened version of \cite[Proposition 3.3]{MPR10}. To apply Theorem~\ref{thm:maxStochConv} to the semigroup generated by an R-sectorial operator satisfying the conditions in (iii)
of Example~\ref{exa:dilation}, one should note that the space $\gamma(L^2(\R),X)$ of all $\gamma$-radonifying operators from $L^2(\R)$ to $X$ is isometrically isomorphic to a closed subspace of $L^2(\tilde{\Om};X)$, for a suitable probability space $\tilde{\Om}$. Therefore, if $X$ has (martingale) type $s$, then $\gamma(L^2(\R),X)$ has (martingale) type $s$ as well.

\section{The Poisson stochastic integral in UMD Banach spaces}\label{sec:UMD}

As mentioned before, the estimates in Theorems~\ref{thm:typesBanach} and \ref{thm:cotypesBanach} do not lead to an It\^{o} isomorphism if $X$ is not a
Hilbert space. For UMD Banach spaces, however, we can formulate an `abstract' It\^{o}-type isomorphism which can serve as a basis for the development of a vector-valued Poisson Skorohod integral.

In what follows we
let $N$ be a Poisson random measure on a measurable space
$(\YY,\Y)$ with $\sigma$-finite intensity measure $\mu$. We denote
$$ \Y_\mu = \{B\in \Y: \ \mu(B)<\infty\}.$$

\begin{definition}[Poisson $p$-norm]\label{def:poisson-norm}
Let $f : \YY \to X$ be a simple function of the form
\begin{align*}
f = \sum_{j=1}^m \one_{B_j} \otw x_j,
\end{align*}
where $x_j\in X$ and the sets $B_j\in \Y_\mu$ are pairwise disjoint. For $1
\leq p < \infty$ we define the (compensated) \emph{Poisson
$p$-norm} of $f$ by
\begin{align*}
 \| f  \|_{\nu_p(\YY; X)} :=
  \big( \E \| \hN(f)\|_X^p \big)^{\frac1p},
\end{align*}
where
$$ \hN(f):= \sum_{j} \hN(B_j)\otw x_j.$$
\end{definition}
It is a simple matter to check that this definition does not depend on the
particular representation of $f$ as a simple function and that $\|
\cdot\|_{\nu_p(\YY; X)}$ defines a norm on the linear space of simple $X$-valued
functions.\par

A Banach space $X$ is called a {\em UMD space} if for some $p\in (1, \infty)$
(equivalently,
for all $p\in (1,\infty)$) there is a constant $\beta\ge 0$ such that for all
$X$-valued $L^p$-martingale difference sequences $(d_n)_{n\geq 1}$ and all signs
$(\epsilon_n)_{n\geq 1}$ one has
\begin{equation*}
 \E \Big\| \sum_{n=1}^N \epsilon_n d_n\Big\|_X^p \le \beta^p  \E \Big\|
\sum_{n=1}^N  d_n\Big\|_X^p, \quad \forall N\geq 1.
\end{equation*}
The least admissible constant in this definition is called the {\em
UMD$_p$-constant} of $X$ and is
denoted by $\beta_{p,X}$. It is well known that once the UMD$_p$ property holds for
one $p\in (1,\infty)$, then it holds for all
$p\in (1,\infty)$; for proofs see \cite{Bu1, Mau}. For more information on UMD
spaces we refer to the survey papers by
Burkholder \cite{Burk01} and Rubio de Francia \cite{RF}.

\begin{example}
 Every Hilbert space $H$ is a UMD space with $\beta_{p,H} = \max\{p,p'\}$.
\end{example}

\begin{example} The Lebesgue spaces $L^p(\mu)$, $1<p<\infty$, are UMD spaces
(with $\beta_{p,L^p(\mu)} = \max\big\{p,p'\big\}$).
 More generally, if $X$ is a UMD space, then $L^p(\mu;X)$, $1<p<\infty$, is a
UMD space (with $\beta_{p,L^p(\mu;X)} = \beta_{p,X}$)
\end{example}

\begin{example}
Duals, closed subspaces, and quotients of UMD spaces are UMD.
\end{example}

The following result is a special case of the decoupling inequality in \cite[Theorem 13.1]{Nee-ISEM} and
\cite[Theorem 2.4.1]{Ver06}.

\begin{theorem}[UMD Poisson stochastic integral] \label{thm:UMD}
Let $X$ be a UMD space and let $1<p<\infty$. If $\phi$ is a simple
adapted process with values in $X$, then
$$\E \sup_{t>0}\Big\n \int_{(0,t]\times B}
\phi \dd \tilde N \Big\n^p \eqsim_{p,X}
\n \phi \one_{B}\n_{L^p(\Om;\nu_p(\R_+\times J;X))}^p.$$
\end{theorem}
Let us denote by $$L_\mathbb{F}^p(\Om;\nu_p(\R_+\times J;X))$$ the closure
in $L^p(\Om;\nu_p(\R_+\times J;X))$ of the space of all simple adapted
processes in $X$. We will see later in Proposition~\ref{prop:partitions}
that if $X$ is UMD, then $L_\mathbb{F}^p(\Om;\nu_p(\R_+\times J;X))$ is a complemented subspace
of $L^p(\Om;\nu_p(\R_+\times J;X))$.
By the theorem, the stochastic integral extends uniquely to a bounded linear
operator from $L_\mathbb{F}^p(\Om;\nu_p(\R_+\times J;X))$ onto a closed linear
subspace of $L^p(\Om;X)$.\par
The reason for calling the presented It\^{o} isomorphism `abstract' is that
the $\nu_p$-norm is still a stochastic object which may be difficult to calculate
in practice. However, the results in Section~\ref{sec:Poisson-SI} identify the
spaces
$\nu_p(E;X)$ in several important cases. As before, $E$ denotes a $\si$-finite measure space.

\begin{example}
\label{exa:nupHilbert}
If $H$ is a Hilbert space, then it follows from Corollary \ref{cor:Hilbert} that
\begin{align*}
\nu_p(E;H) = \left\{
\begin{aligned}
L^2(E;H)\cap L^p(E;H), & \quad  2\leq p<\infty;\\
L^2(E;H) + L^p(E;H),   & \quad  1<p\le 2;
\end{aligned}\right.
\end{align*}
with equivalent norms. Indeed, let $\hat{N}$ be a Poisson random measure on $\R_+\ti E$. For any $B \in \Y$, we let $N_*(B):=\hat{N}((0,1]\ti B)$. Clearly, if $f:E\rightarrow H$ is simple, then
$$\int_{E} f \ d\tilde{N} \qquad \mathrm{and} \qquad \int_{\R_+ \ti E} \one_{(0,1]} \ot f \ d\tilde{N_*}$$
are identically distributed. Therefore, if $p\geq 2$ then Corollary \ref{cor:Hilbert} (with $J=E$) implies that
$$\|f\|_{\nu_p(E;H)} \eqsim_p \|\one_{(0,1]} \ot f\|_{\cD^p_{2,H}\cap\cD^p_{p,H}} = \|f\|_{L^2(E;H)\cap L^p(E;H)}.$$
If $1<p\leq 2$, then we obtain
$$\|f\|_{\nu_p(E;H)} \eqsim_{p} \|\one_{(0,1]} \ot f\|_{\cD^p_{2,H} + \cD^p_{p,H}} \leq \|f\|_{L^2(E;H) + L^p(E;H)}.$$
The reverse estimate follows by the duality argument in the proof of Theorem~\ref{thm:cotypesBanach}.
\end{example}

\begin{example}\label{ex:Lq}
Let $X = L^q(S)$ with $1<q<\infty$. By remarks similar to the ones made in Example~\ref{exa:nupHilbert},
Theorem \ref{thm:summarySILqPoisson} implies that $\nu_p(E;L^q(S))$ is given by
\begin{align*}
L^q(S;L^2(E))\cap L^q(E;L^q(S)) \cap L^p(E;L^q(S)), & \quad  2\leq q\leq p;\\
L^q(S;L^2(E)) \cap (L^q(E;L^q(S)) + L^p(E;L^q(S))), & \quad  2\leq p\leq q;\\
(L^q(S;L^2(E)) \cap L^q(E;L^q(S))) + L^p(E;L^q(S)), & \quad  p\le 2\leq q;\\
(L^q(S;L^2(E)) + L^q(E;L^q(S))) \cap L^p(E;L^q(S)), & \quad  q\le 2\leq p;\\
L^q(S;L^2(E)) + (L^q(E;L^q(S)) \cap L^p(E;L^q(S))), & \quad  q\leq p\leq 2;\\
L^q(S;L^2(E)) + L^q(E;L^q(S)) + L^p(E;L^q(S)), & \quad  p\leq q\leq 2;
\end{align*}
with equivalent norms.
\end{example}

For a general Banach space $X$, Theorems~\ref{thm:typesBanach} and \ref{thm:cotypesBanach}
imply two continuous inclusions for $\nu_p(E;X)$:
if $X$ has martingale type $1<s\leq 2$, then
\begin{equation}\label{eq:incl-mt}
\left.
  \begin{array}{ll}
    L^s(E;X)  \cap L^p(E;X), & s\leq p \\
    L^s(E;X) +     L^p(E;X), & p\leq s
  \end{array}
\right\}
\hookrightarrow \nu_p(E;X),
\end{equation}
and if $X$ has martingale cotype $2\leq s<\infty$, then
\begin{equation}\label{eq:incl-mc}
\nu_p(E;X))\hookrightarrow
\left\{
  \begin{array}{ll}
    L^s(E;X) \cap L^p(E;X), & s\leq p; \\
    L^s(E;X) +    L^p(E;X), & p\leq s.
  \end{array}
\right.
\end{equation}
In particular, since any UMD space has finite martingale cotype, we see that if $X$ is a UMD space, then every element in $\nu_p(E;X)$ can be identified with
an $X$-valued function on $E$.

In the next section, we shall use Theorem \ref{thm:UMD} to identify the UMD
Poisson stochastic integral as a special instance of the Poisson Skorohod
integral. Then, in the final Section \ref{sec:CO}
we will prove a Clark-Ocone type representation theorem for the UMD Poisson
stochastic integral.

\section{The Malliavin derivative}\label{sec:Malliavin}

In the scalar-valued case there are various ways to extend the classical
Malliavin calculus to the Poisson case. Very complete results can be found in
the recent paper of
Last and Penrose \cite{LaPe11a}, to which we refer the reader for further
references to this extensive subject.

Here we wish to extend the ideas developed in the previous sections into a
vector-valued Poisson Malliavin calculus. For this purpose we shall adopt an
approach
which stays close to the standard approach in the Gaussian case as presented,
for example, in Nualart's book \cite{Nua}, in that we define a Poisson Malliavin
derivative directly in terms of a class of cylindrical functions associated to a
Poisson random measure. In doing so we can essentially follow the lines of
vector-valued Malliavin calculus in the Gaussian case as developed in
\cite{Maa10,MvNCO}.

We consider a probability space $(\Om, \F, \P )$, and a Poisson random measure
$N$ defined on a measurable space $(\YY,\Y)$ with $\sigma$-finite intensity measure $\mu$.
We shall use the notation $[\cdot, \cdot]$ to denote the inner product in $L^2(\YY,\mu)$.

It will be useful to employ the following multi-index notation.
For a tuple $\BB = (B_1, \ldots, B_M )$  in $\Y_\mu \times
\ldots \times \Y_\mu$ we set
\begin{align*}
 N(\BB) := (N(B_1), \ldots, N(B_M) )\,\qquad
 \mu(\BB) := (\mu(B_1), \ldots, \mu(B_M) ).
\end{align*}
Let $\ee_m := (0, \ldots, 0,  1, 0 , \ldots, 0)$ denote the $m$-{th} unit vector
in $\R^M$.
We shall also write $\jj := (j_1, \ldots, j_M)$ and $\kk := (k_1, \ldots, k_M)$,
etc. We use multi-index notation
\begin{align*}
  \kk! := k_1! \cdot \ldots \cdot k_M!\, \qquad
  \rr^\kk := r_1^{k_1} \cdot \ldots \cdot r_M^{k_M}\,
\end{align*}
for $\rr := (r_1, \ldots, r_M) \in \R_+^M$.

First we define a suitable class of functions to work with.
For the remainder of this section we fix an arbitrary Banach space $X$.

\begin{definition}[Cylindrical functions]\label{def:cylindrical}
A {\em cylindrical function} is a random variable $F : \Om \to \R$ of the form
\begin{align} \label{eq:cylindrical}
  F = f(N(\BB)),
\end{align}
where $M \in \N$, $f: \N^M \to \R$, and $\BB = (B_1, \ldots,B_M)$
with $B_1, \ldots, B_M \in \Y_\mu$.
\end{definition}
The real vector space of all cylindrical functions is denoted by $\cC$.
We denote by $\cCX = \cC \otimes X$ the collection of all vector-valued random variables $F : \Om
\to X$ of the form
\begin{align}\label{eq:cylindrical-X}
 F = \sum_{i = 1}^n F_i \ot x_i\,
\end{align}
where $n \geq 1$, $F_i \in \cC$, and $x_i \in X$ for $i =1, \ldots, n$.
The elements of $\cCX$ will be called {\em $X$-valued cylindrical functions}.

\begin{remark}\label{rem:disjoint}
In the sequel, when taking a function $F \in \cC$ of the form
\eqref{eq:cylindrical}, we will always assume (possibly without explicit
mentioning) that the sets $B_1, \ldots, B_M$ are pairwise disjoint. Clearly,
this does not yield any loss of generality.
\end{remark}

\begin{definition}[Malliavin derivative]\label{def:Malliavin-derivative}
For a random variable $F:\Om\rightarrow X$ of the form \eqref{eq:cylindrical-X},
the \emph{Malliavin derivative} $D F : \Om \to L^2(\YY)\otimes X$ is defined by
\begin{align*}
  D F := \sum_{i=1}^n\sum_{m=1}^M \Big( f_i(N(\BB) + \ee_m) - f_i(N(\BB)) \Big) \ot
       (\one_{B_m} \otimes x_i).
\end{align*}
\end{definition}

It is easy to see that this definition does not depend on the particular representation
of $F$.
It should be compared to the one in \cite{LaPe11a},
where an analogous construction was given on Poisson space. In the scalar-valued setting,
an alternative (and equivalent) definition of the Malliavian derivative can be given in terms
of a Fock space construction; for more details we refer to \cite{Lok, NuVi}.

The following identity from Poisson stochastic analysis is well-known in the scalar-valued
case (cf. \cite[Eq. (I.23)]{DKW}).
For the convenience of the reader we supply a short proof.

\begin{proposition}\label{prop:product}
 For all $F \in \cCX$ and $G \in \cCXs$ we have $\ip{F,G} \in \cC$ and
\begin{align*}
 D\ip{F,G} = \ip{DF, G} + \ip{F, DG} + \ip{DF, DG}.
\end{align*}
\end{proposition}

\begin{proof}
It suffices to show that
\begin{align*}
D(FG) = (DF)G + F(DG) + DF\cdot DG
\end{align*}
for all $F,G \in \cC$.
We may thus assume that $F, G \in \cC$ are of the form \eqref{eq:cylindrical}
with $B_1, \ldots, B_M$ pairwise disjoint, and that $\bigcup_{m = 1}^M B_m =
\YY$,
say
\begin{equation*}
\begin{aligned}
  F & = f(N(\BB)) =  f(N(B_1), \ldots, N(B_M)),\\
  G & = g(N(\BB)) =  g(N(B_1), \ldots, N(B_M)).
\end{aligned}
\end{equation*}
Then
\begin{align*}
D(FG) &  = \sum_{m=1}^M \Big(fg(N(\BB) + \ee_m) - fg(N(\BB)) \Big) \ot
\one_{B_m} \\
& = \sum_{m=1}^M \Big(\big[ f(N(\BB) + \ee_m) - f(N(\BB))\big] g(N(\BB)+\ee_m)
\Big) \ot \one_{B_m} \\
& \qquad + \sum_{m=1}^M \Big( f(N(\BB))\big[ g(N(\BB)+\ee_m) - g(N(\BB))
\big]\Big) \ot \one_{B_m} \\
& = \sum_{m=1}^M \Big(\sum_{n=1}^M \big[ f(N(\BB) + \ee_n) - f(N(\BB))\big]
g(N(\BB)+\ee_m) \Big) \ot \one_{B_m}\one_{B_n} \\
& \qquad + \sum_{m=1}^M \Big( f(N(\BB))\sum_{n=1}^M\big[ g(N(\BB)+\ee_n) -
g(N(\BB)) \big]\Big) \ot \one_{B_m}\one_{B_n} \\
& = \Big(\sum_{m=1}^M  g(N(\BB)+\ee_m)  \ot \one_{B_m}\Big)\cdot DF +
\Big(\sum_{m=1}^M  f(N(\BB)) \ot \one_{B_m}\Big)\cdot DG \\
& = \Big(\sum_{m=1}^M  [g(N(\BB)+\ee_m)-g(N(\BB)) ] \ot \one_{B_m}\Big)\cdot DF
\\
& \qquad +  \sum_{m=1}^M  g(N(\BB)) \ot \one_{B_m}\cdot DF + \sum_{m=1}^M
f(N(\BB)) \ot \one_{B_m}\cdot DG
\\ & =  DG\cdot DF + G(DF) + (DG)F.\end{align*}
\end{proof}

The essential ingredient to prove closability of the Malliavin derivative is an integration by parts formula:

\begin{proposition}[Integration by parts]\label{prop:ibp}
For $F \in \cCX$ and $B \in \Y_\mu$ we have
\begin{align*}
  \E \int_B DF \dd \mu = \E\big( \hN(B)F  \big).
\end{align*}
\end{proposition}

\begin{proof}
It suffices to consider the scalar-valued case. Thus let $F\in \cC$ be of the form \eqref{eq:cylindrical} with $B_1, \ldots,
B_M$ pairwise disjoint. Set $J_m := B_m \cap B$ and $K_m  := B_m \setminus B$
and write $\JJ := ( J_1, \ldots, J_M )$ and $\KK := ( K_1, \ldots, K_M )$.
Using that the random variables $N(J_1), \ldots, N(J_M), N(K_1), \ldots, N(K_M)$ are independent
and Poisson distributed, we obtain
\begin{align*}
&  \E \int_B DF \dd \mu
\\& = \sum_{m=1}^M  [\one_{B_m}, \one_{B}]  \E
   \Big( f(N(\BB) + \ee_m) - f(N(\BB)) \Big)
\\& = \sum_{m=1}^M  \mu({B_m}\cap{B})  \E
   \Big( f(N(\JJ) + N(\KK) + \ee_m) - f(N(\JJ) + N(\KK)) \Big)
\\&   = \sum_{m=1}^M \mu(J_m)
\exp\Big( -\sum_{l=1}^M \mu(J_l) + \mu(K_l) \Big)
\\& \ \ \times
\sum_{\jj,\kk \in \N^M} \frac{\mu(\JJ)^{\jj}}{\jj!} \frac{\mu(\KK)^{\kk}}{\kk!}
    \Big( f(\jj + \kk + \ee_m) - f(\jj + \kk ) \Big)
\\&   = \sum_{m=1}^M
\exp\Big( -\sum_{l=1}^M \mu(J_l) + \mu(K_l) \Big)
\\& \ \  \times\!\!\!
\sum_{\jj,\kk \in \N^M}\!\! \Big(
 \frac{\mu(\JJ)^{\jj + \ee_m}}{(\jj + \ee_m)!} \frac{\mu(\KK)^{\kk}}{\kk!}
    (j_m +1)  f(\jj + \kk + \ee_m) )
-\frac{\mu(\JJ)^{\jj}}{\jj!} \frac{\mu(\KK)^{\kk}}{\kk!}
  \mu(J_m)   f(\jj + \kk)\Big).
\end{align*}
Replacing $j_m$ by $j_m - 1$ in the first summation, one obtains
\begin{align*}
&  \E \int_B DF \dd \mu
\\&   = \sum_{m=1}^M
\exp\Big( -\sum_{l=1}^M \mu(J_l) + \mu(K_l) \Big)
\\& \quad \times
\sum_{\jj,\kk \in \N^M} \Big(
 \frac{\mu(\JJ)^{\jj}}{\jj!} \frac{\mu(\KK)^{\kk}}{\kk!}
    j_m  f(\jj + \kk) )
-\frac{\mu(\JJ)^{\jj}}{\jj!} \frac{\mu(\KK)^{\kk}}{\kk!}
  \mu(J_m)   f(\jj + \kk)
\Big)
\\&   = \E \sum_{m=1}^M
 \Big( N(J_m) f(N(\JJ) + N(\KK)) )
- \mu(J_m)   f(N(\JJ) + N(\KK))
\Big)
\\&   = \E \sum_{m=1}^M
 \hN(J_m) f(N(\BB) )
\\&   = \E
 \hN\Big(\bigcup_{m=1}^M J_m\Big) f(N(\BB) )
.
\end{align*}
On the other hand, since $R := B \setminus \bigcup_{m=1}^M J_m$ does not
intersect $\bigcup_{m=1}^M B_m$, it follows that $N(R)$ and $f(N(\BB))$ are
independent, hence
\begin{align*}
 \E \big( \hN(R) f(N(\BB)) \big)
   =  \E \hN(R) \; \E f(N(\BB))
   = 0.
\end{align*}
We infer that
\begin{align*}
  \E \int_B DF \dd \mu
    = \E
 \hN\Big(R \cup \bigcup_{m=1}^M J_m\Big) f(N(\BB) )
    = \E
 (\hN(B) F) .
\end{align*}
\end{proof}

\begin{theorem}[Closability] The operator $D: \cC\ot X \to \cC\otimes L^2(\YY)\ot X$
is closable from $L^p(\Omega;X)$ into $L^p(\Omega;\nu_p(\YY;X))$ for all $1<
p<\infty$.
\end{theorem}
\begin{proof}
The proof is based upon the fact that if $Y$ is a Banach space, $Z\subseteq Y^*$
is a weak$^*$-dense linear subspace,
and $f\in L^p(\Omega;Y)$ is a function
such that $\E \lb f,\zeta\rb = 0$ for all $\zeta\in Z$,
then $f = 0$ in $L^p(\Omega;Y)$. We apply this to $Y  = \nu_p(\YY;X)$ and $Z$
the linear span of
the $X^*$-valued indicator functions $\one_B\otimes x^*$ with $B\in \Y_\mu$
and $x^*\in X^*$.

Fix $1< p<\infty$ and let $(F_n)$ be a sequence in $\cCX$ such that $F_n \to 0$
in
$L^p(\Om;X)$ and $DF_n \to G$ in $L^p(\Om;\nu_p(\YY;X))$ as $n\to\infty$.
We must prove that $G=0.$

For each $B\in \Y_\mu$ and $x^*\in X^*$,
using Proposition \ref{prop:ibp} and the fact that $\hN(B)\in L^q(\Om)$,
$\frac1p+\frac1q=1$, we obtain
 \begin{align*}
 \E \int_B \lb G,x^*\rb \dd \mu
&  = \lim_{n \to \infty}  \E \int_B \lb DF_n,x^*\rb  \dd \mu
 = \lim_{n \to \infty} \E (\hN(B){\lb F_n, x^*\rb}) =0.
 \end{align*}
This being true for all $B\in\Y_\mu$ and $x^*\in X^*$, we conclude that $G =
0$.
\end{proof}

By routine arguments one establishes the following density result.
We denote by $$\cG := \sigma(\tilde N(B): \ B\in \Y_\mu\}$$ the $\sigma$-algebra
in $\Om$ generated by $\tilde N$.

\begin{lemma}\label{lem:cCdense}
 For all $1\le p<\infty$, $\cCX$ is dense in $L^p(\Om,\cG;X)$.
\end{lemma}

Thanks to this lemma, the closure of $D$ is densely defined as an operator
from $L^p(\Om,\cG;X)$ into $L^p(\Om;\nu_p(E;X))$.
With slight abuse of notation we will denote this closure by $D$ again,
or, if we want to be more precise, by $D_p^X$.
The dense domain of this closure in $L^p(\Om,\cG;X)$ is denoted by
$\dD^{1,p}(\Om;X).$ This is a Banach space endowed with the norm
 \begin{align*}
 \|F\|_{\dD^{1,p}(\Om;X)} := \Big( \| F \|_{L^p(\Om;X)}^p
           + \| DF \|_{L^p(\Om;\nu_p(\YY;X))}^p \Big)^{\frac1p}.
 \end{align*}

\subsection{The Skorohod integral}
As in the Gaussian case, the adjoint of the Poisson Malliavin derivative extends
the It\^o stochastic integral in a natural way. This anticipative extension
of the stochastic integral, the
so-called Skorohod integral, has been studied, in the Poisson setting, by
many authors \cite{CaPa, DKW, Kab, LaPe11a, NuVi}. Here we will show that if $X$ is
a UMD space, the adjoint of the operator $D$ introduced above extends the
It\^o integral of Section \ref{sec:UMD}.

From now on we assume that $X$ is a UMD space.
We begin by defining the divergence operator as the adjoint of the Malliavin derivative.
It thus acts on random variables taking
values in the dual space of $\nu_p(\YY;X)$.

\begin{definition}
Let $1 < p < \infty$ satisfy $\frac1p + \frac1q = 1$. The \emph{divergence operator}
\begin{align*}
\delta = \delta_q^{X^*} : L^q(\Om;\nu_q(\YY;X^*) )
  \to  L^q( \Om,\cG; X^* )
\end{align*}
is defined to be the adjoint of the operator
\begin{align*}
D = D_p^{X} : L^p(\Om,\cG;X) \to L^p(\Om;\nu_p(\YY;X)).
\end{align*}
\end{definition}

A word of explanation is needed at this point. The mapping
$$\sum_{n=1}^N \one_{B_n}\otimes x_n\mapsto
\sum_{n=1}^N \tilde N(B_n)\otimes x_n,$$ where the sets $B_n\in\Y_\mu$ are disjoint,
identifies $\nu_p(E;X)$ isometrically with a closed subspace of $L^p(\Om;X)$.
With this identification, $D = D_p^{X}$ defines a densely defined and closed linear
operator from $ L^p(\Om,\cG;X)$ into $ L^p(\Om;L^p(\Om;X))$.
Since $X$ is UMD and therefore reflexive, the duals of $L^p(\Om,\cG;X)$ and $ L^p(\Om;L^p(\Om;X))$ may be
identified with $L^q(\Om,\cG;X^*)$ and $L^q(\Om;L^q(\Om;X^*))$. The adjoint operator
$D^*$ is then a densely defined closed linear operator from $ L^q(\Om;L^q(\Om;X^*))$
to  $L^q(\Om,\cG;X^*)$. We now define $\delta = \delta_q^{X^*}$ as the restriction of
$D^*$ to $L^q(\Om;\nu_q(\YY;X^*) )$. That this restricted
operator is again densely defined will follow from the next lemma.

\begin{lemma} \label{lem:divelementary}
Let $F = \sum_{i} F_i \ot x_i^* \in \cCXs$ be as in \eqref{eq:cylindrical-X} and
let $B\in \Y_\mu$ be disjoint with the sets used in the representation of
the $F_i$'s.
For all $1<p<\infty$ we have $\one_B F \in \Dom(\d_p^{X^*})$ and
$$ \d(\one_B F) = \hN(B)F.$$
In particular, $\delta$ is densely defined.
 \end{lemma}
\begin{proof}
Let $G\in \cC\ot X$. Using Propositions \ref{prop:product} and \ref{prop:ibp} and the fact that
$\one_{B} DF =0$, we have
\begin{align*}
\E \int_E\lb DG,\one_B F \rb \dd \mu & = \E \int_B ( D\lb G,F\rb - \lb G,DF\rb -
\lb DG,DF\rb)\dd\mu \\
                       & =  \E\int_B  D\lb G,F\rb \dd\mu
                         = \E(\hN(B)\lb G,F\rb).
\end{align*}
\end{proof}
Since $X$ is UMD and therefore reflexive, we can regard the divergence operator as an operator
\begin{align*}
\delta = \delta_p^{X} : L^p(\Om;\nu_p(\YY;X) )
  \to  L^p(\Om,\cG; X ),
\end{align*}
by considering the adjoint of $D_q^{X^*}$ when $\frac1p + \frac1q = 1$ and using
the identification $X = X^{**}$.

We return to the situation considered in Section \ref{sec:Poisson-SI}
and take $$E = \R_+\times J,$$ where $(J,\cJ,\nu)$ is a $\si$-finite measure space.
Let $N$ be a Poisson
random measure on $(\R_+\ti J, \cB(\R_+)\ti \cJ,dt\ti \nu)$.
As before we let $\mathbb{F} = (\cF_t)_{t>0}$ be the filtration generated by $\tilde N$, i.e.,
$$\cF_t =\sigma\{\tilde{N}((s,u]\ti A) \ : \ 0\leq s<u\leq t, \ A \in \cJ\}.$$
We will show that in this setting
the divergence operator $\d$ is an extension of the Poisson stochastic integral
$I= I_p^X: L^p(\Om;\nu_p(\R_+\ti J;X))\to L^p(\Om,\cG;X)$.

We recall from Section \ref{sec:UMD} that $L^p_{\mathbb{F}}(\Om;\nu_p(\R_+\ti J;X))$
denotes the closure of all simple, adapted processes in
$L^p(\Om;\nu_p(\R_+\ti J;X))$.
The following result shows that the divergence
operator $\delta$ coincides with the It\^{o} integral for adapted integrands,
and hence can be viewed as a Skorokhod integral for non-adapted processes.

\begin{theorem}
\label{thm:Skorokhod}
Let $1<p<\infty$ and let $X$ be a UMD space. Then $L^p_{\mathbb{F}}(\Omega;\nu_p(\R_+\ti J;X))$ is
contained in $\mathrm{D}(\delta_p^X)$ and
\begin{equation}
\label{eqn:Skorokhod}
\delta(\phi) = I(\phi) \ \ \hbox{for all } \phi \in L^p_{\mathbb{F}}(\Omega;\nu_p(\R_+\ti J;X)).
\end{equation}
\end{theorem}
\begin{proof}
Suppose first that $\phi$ is a simple, adapted process of the form (\ref{eqn:simpleBanach}),
with $F_{ijk} \in \mathscr{C}(\Omega,\cF_{t_i})$ for
all $i,j,k$. By Lemma~\ref{lem:divelementary}, $\phi \in \mathrm{D}_p(\delta)$ and
\begin{align*}
\delta(F_{ijk} \one_{(t_i,t_{i+1}]} \one_{A_j} x_{ijk}) & = F_{ijk}
\tilde{N}((t_{i},t_{i+1}]\times A_j) x_{ijk} \\
& =  I(F_{ijk} \one_{(t_i,t_{i+1}]} \one_{A_j} x_{ijk}),
\end{align*}
so by linearity (\ref{eqn:Skorokhod}) holds. Since $\mathscr{C}(\Omega,\cF_{t_i})$
is dense in $L^p(\Om,\cF_{t_i})$ by lemma~\ref{lem:cCdense} and $\delta$ is closed,
we conclude that any simple,
adapted process $\phi$ is in $\mathrm{D}_p(\delta)$ and $\delta(\phi)=I(\phi)$. Finally,
by density of the simple, adapted processes in $L^p_{\mathbb{F}}(\Omega;\nu_p(\R_+\ti J;X))$
 we find that $L^p_{\mathbb{F}}(\Omega;\nu_p(\R_+\ti J;X))\subseteq \mathrm{D}_p(\delta)$
and (\ref{eqn:Skorokhod}) holds.
\end{proof}
\begin{remark}
It is important to emphasize that $\bF$ denotes the natural filtration generated
by $\tilde{N}$. Indeed, in the proof of Theorem~\ref{thm:Skorokhod} we use that
$\mathscr{C}(\Omega,\cF_{s})$ is dense in $L^p(\Om,\cF_{s})$ for any $s\ge 0$.
\end{remark}

\section{A Clark-Ocone formula}\label{sec:CO}

In scalar-valued Poisson stochastic calculus, Clark-Ocone type representation
theorems representing a $\cG$-measurable random variable as the stochastic
integral of an adapted process defined in terms of its Malliavin derivative have been obtained in
various degrees of generality by many authors. We mention \cite{DOP,
LaPe11b, Pri94, Wu} and refer the reader to these works for further
bibliographic references. All these papers are concerned with the real-valued
case. To the best of our knowledge, the Banach space-valued case has not been considered yet in the literature.
Here we shall present an extension of the Clark-Ocone theorem to the
UMD space-valued Poisson stochastic integral of Section \ref{sec:UMD}.

Following the approach of \cite{MvNCO}, our first aim is to construct a
projection $\cP_{\mathbb{F}}$ in the space $L^p(\Om;\nu_p(\R_+ \times {J};X))$
onto the subspace $L_{\mathbb{F}}^p(\Om;\nu_p(\R_+ \times {J};X))$ introduced in Section \ref{sec:UMD}. Formally
this projection is given by $$\cP_{\mathbb{F}} F(t) = \E(F(t)|\F_t).$$
The main issue is to give a rigorous interpretation of this formula in the
present context. For this purpose we shall need a Poisson analogue of the notion
of R-boundedness.

\begin{definition}[$\nu_p$-Boundedness]\label{def:nu-bdd}
Let $1 \leq p < \infty$ and $X,Y$ be Banach spaces. A collection of bounded linear operators $\mathscr{T}
\in \cL(X,Y)$
is said to be {\em $\nu_p$-bounded} if there exists a constant $C > 0$ such that
the estimate
\begin{align*}
   \Big( \E \Big\| \sum_{j} \hN(B_j) T_j x_j  \Big\|_Y^p \Big)^{\frac1p}
   \leq C
   \Big( \E \Big\| \sum_{j} \hN(B_j) x_j  \Big\|_X^p \Big)^{\frac1p}
\end{align*}
holds for every finite collection of pairwise disjoint sets $B_j \in \Y_\mu$, every finite sequence
$x_j\in X$ and every finite sequence $T_j\in \mathscr{T}$.
\end{definition}

If we replace the random variables $\hN(B_j)$ by independent Rademacher variables we
obtain the related notion of {\em $R$-boundedness}. In this case the smallest
admissible constant $C$ is called the {\em $R$-bound}.

\begin{proposition}\label{prop:R-nu}
Let $1 \leq p < \infty$. If a collection of bounded linear operators $\mathscr{T} \in \cL(X,Y)$
is $R$-bounded with $R$-bound $R(\mathscr{T})$, then it is $\nu_p$-bounded as well.
\end{proposition}

\begin{proof}
Let $(\eps_j)_{j\geq 1}$ be a Rademacher sequence. By symmetrization \cite[Lemma 6.3]{LeT91},
\begin{align*}
 \E \Big\| \sum_{j} \hN(B_j) T_j x_j  \Big\|_Y^p
 & \eqsim   \E\E_\eps \Big\| \sum_{j} \eps_j \hN(B_j)  T_j x_j \Big\|_Y^p
 \\& \leq R(\mathscr{T})^p   \E\E_\eps  \Big\| \sum_{j} \eps_j \hN(B_j) x_j
\Big\|_X^p
 \\&  \eqsim R(\mathscr{T})^p
 \E \Big\| \sum_{j} \hN(B_j) x_j  \Big\|_X^p.
\end{align*}
\end{proof}

\begin{example}[Conditional expectations]\label{ex:condExp}
Let $X$ be a UMD Banach space and let $(\F_t)_{t \geq 0}$ be a filtration on a
probability space $(\Om, \cF, \P)$. For all $1 \leq p < \infty$ and all $1 < q <
\infty$ the collection of conditional expectation operators
\begin{align*}
 \{ \E(\cdot | \F_t) : t \geq 0 \}
\end{align*}
is $\nu_p$-bounded on $L^q(\Om; X)$.
Indeed, by Bourgain's vector-valued Stein inequality
\cite{Bou} (see also \cite{CPSW}), this collection is $R$-bounded on $L^q(\Om; X)$.
\end{example}

We can now rigorously define the projection
$\cP_{\mathbb{F}}$. Let $\pi = \{t_1, \ldots, t_{N+1}\}$
with $0 = t_1 \leq t_2 \ldots \leq t_{N+1} <\infty$ be a partition.
Let $F \in L^p(\Om;\nu_p(\R_+ \times {J};X))$ be a simple process of the form
\begin{align*}
  F = \one_{(a, b]} \otw f \otw \one_A \otw x,
\end{align*}
where $0 \leq a < b <\infty$, $f \in L^\infty(\Om)$,
$A\in \Y_\mu$,
and $x \in
X$. We define $\cP_{\mathbb{F}}^\pi F \in  L^p_{\mathbb{F}}(\Om;\nu_p(\R_+
\times {J};X))$ by
\begin{align*}
 \cP_{\mathbb{F}}^\pi F = \sum_{n=1}^N
 			\one_{(a, b] \cap (t_{n}, t_{n+1}]} \otw
\E(f|\cF_{t_{n}})
			  \otw \one_A \otw x
\end{align*}
and extend this definition by linearity.

\begin{proposition}\label{prop:partitions}
Let $X$ be a UMD space. For each partition $\pi$ the mapping
$\cP_{\mathbb{F}}^\pi$ has a unique extension to a bounded projection in the
space $L^p(\Om;\nu_p(\R_+ \times {J};X))$. These projections are uniformly
bounded and
\begin{align*}
 \cP_{\mathbb{F}} F
   =  \lim_{\pi} \cP_{\mathbb{F}}^\pi F
\end{align*}
defines a bounded projection in $L^p(\Om;\nu_p(\R_+ \times {J};X))$ onto
$L^p_{\mathbb{F}}(\Om;\nu_p(\R_+ \times {J};X))$.
\end{proposition}

The above limit is taken along the net of all partitions $\pi$ of $\R_+$, which are partially ordered by refinement.

\begin{proof}
By Example \ref{ex:condExp} the collection of conditional expectation operators
$\E(\cdot|\cF_t)$ is $\nu_p$-bounded on $L^p(\Om; X)$.
Hence for a simple process of the form
\begin{align*}
F = \sum_{j=1}^J\sum_{k=1}^K \sum_{l=1}^L
 f_{jkl} \otw  \one_{(s_{j}, s_{j+1}]} \otw\one_{A_l} \otw x_{jkl},
\end{align*}
with the sets $A_l\in\Y_\mu$ pairwise disjoint,
we have
\begin{align*}
& \| \cP_{\mathbb{F}}^\pi F \|_{\nu_p(\R_+ \times J; L^p(\Om;X))}
\\&   = \| \sum_{n=1}^N \sum_{j,k,l}
\one_{(s_{j}, s_{j+1}] \cap (t_{n}, t_{n+1}]}
		\otw \E(f_{jkl} | \cF_{t_{n}}) \otw  \one_{A_l} \otw x_{jkl}
			   \|_{\nu_p(\R_+ \times J; L^p(\Om;X))}
\\&   \lesssim_{p,X} \| \sum_{n=1}^N \sum_{j,k,l}
\one_{(s_{j}, s_{j+1}] \cap (t_{n}, t_{n+1}]}
		\otw   f_{jkl} \otw  \one_{A_l} \otw x_{jkl}
			   \|_{\nu_p(\R_+ \times J; L^p(\Om;X))}
\\&   = \|  \sum_{j,k,l}
\one_{(s_{j}, s_{j+1}]}
		\otw   f_{jkl} \otw  \one_{A_l} \otw x_{jkl}
			   \|_{\nu_p(\R_+ \times J; L^p(\Om;X))}
\\&   = \| F \|_{\nu_p(\R_+ \times J; L^p(\Om;X))}.
\end{align*}
By Fubini's Theorem we have an isometric isomorphism
\begin{align*}
 \nu_p(\R_+ \times J; L^p(\Om;X)) \eqsim
 L^p(\Om; \nu_p(\R_+ \times J; X)).
\end{align*}
As a consequence, the operator $\cP_{\mathbb{F}}^\pi$ has a unique extension to
a bounded operator on $L^p(\Om; \nu_p(\R_+ \times J; X))$, with norm bounded
from above by a constant depending only on $p$ and $X$. Obviously, this operator
is a projection. Moreover, if $\pi' \subseteq \pi$, then
\begin{align*}
 \cP_{\mathbb{F}}^{\pi} \circ \cP_{\mathbb{F}}^{\pi'}
   =  \cP_{\mathbb{F}}^{\pi'}.
\end{align*}
This implies that the net $(\cP_\pi^{\mathbb{F}})_\pi $ is upward directed.
Since it is also uniformly bounded, the strong operator limit
$$\cP_{\mathbb{F}}:= \lim_\pi \cP_{\mathbb{F}}^{\pi} = \bigvee_\pi
\cP^\pi_{\mathbb{F}}$$ exists in $L^p(\Om; \nu_p(\R_+ \times J; X))$ and
defines a projection onto
$$\overline{\bigcup_{\pi}\cP_{\bF}^{\pi}L^p(\Om;\nu_p(\R_+\ti J;X))}
= L^p_{\mathbb{F}}(\Om;\nu_p(\R_+ \times {J};X)).$$
\end{proof}

The following lemma will be
useful in the proof of Theorem \ref{thm:ClarkOcone}.
 \begin{lemma} \label{lem:itoduality}
let $X$ be a UMD space and let
$1 < p,q < \infty$ satisfy
$\frac1p+\frac1q=1$. For all random variables $\Phi \in
L^p_{{{\mathbb{F}}}}(\Om;\nu_p(\R_+\times{J}; X))$ and $\Psi \in
L^q_{{{\mathbb{F}}}}(\Om;\nu_q(\R_+\times{J}; X^*))$ we have
 \begin{align*}
 \E\ip{I_p(\Phi),I_q(\Psi)}  = \E\lb \Phi,\Psi \rb.
 \end{align*}
 \end{lemma}

  \begin{proof}
When $\Phi$ and $\Psi$ are simple adapted processes the result follows
by direct computation. The general case follows by approximation.
  \end{proof}

We can now prove a Clark-Ocone representation formula.

\begin{theorem}[Clark-Ocone representation]\label{thm:ClarkOcone}
Let $X$ be a UMD space and $1 < p < \infty$. If $F \in
\dD^{1,p}(\Om;X)$ is $\cG$-measurable, then
 \begin{align*}
 F = \E(F) + I (\cP_{\mathbb{F}}(D F)).
 \end{align*}
Moreover, $\cP_{\mathbb{F}}(D F)$ is the unique $Y \in
L_{\mathbb{F}}^p(\Om;\nu_p(\R_+\times J; X))$ satisfying $F =
\E(F) + I(Y)$.
\end{theorem}

\begin{proof}
We may assume that $\E(F) = 0$ as $D(\E(F))=0$. Suppose first that there exists
$\phi \in L_{\mathbb{F}}^p(\Om;\nu_p(\R_+ \times J;X))$
such that
\begin{align}\label{eq:representation}
F = I(\phi).
\end{align}
Let $\frac1p+\frac1q= 1,$ and let $\tilde \phi\in L^q(\Om;\nu_q(\R_+ \times J;X^*))$
be arbitrary. Arguing as in \cite{MvNCO}, we obtain using Theorem \ref{thm:Skorokhod},
and Lemma \ref{lem:itoduality},
\begin{align*}
    \E\ip{ \cP_{{{\mathbb{F}}}}(DF),\tilde \phi}
& =  \E \ip{ DF,\cP_{{{\mathbb{F}}}} \tilde \phi}
 =  \E\ip{ F,\d(\cP_{{{\mathbb{F}}}} \tilde \phi)}
\\ & =  \E\ip{I(\phi),I(\cP_{{{\mathbb{F}}}} \tilde \phi)}
\\ & =  \E \ip{ \phi,\cP_{{{\mathbb{F}}}} \tilde \phi}
 =  \E\ip{ \cP_{{{\mathbb{F}}}}\phi, \tilde \phi}
 =  \E\ip{\phi, \tilde \phi}.
\end{align*}
Since this holds for all $\tilde \phi \in
 L^q(\Om;\nu_q(\R_+ \times J;X^*)),$ it follows that $\phi =
\cP_{{{\mathbb{F}}}}(DF)$ and therefore
\begin{align}\label{eq:FIG}
 F = I(\phi) = I(\cP_{{{\mathbb{F}}}}(DF)).
\end{align}
Next let $F \in\dD^{1,p}(\Om;X)$ be arbitrary with $\E(F) = 0$.
By density and linearity it suffices to prove (\ref{eq:FIG}) for a function $F=G\ot x$,
where $G \in L^2(\Om,\cG)\cap L^p(\Om,\cG)$. We need to show that (\ref{eq:representation}) holds. In view of the identity  $I(\phi\ot x) = I(\phi)\ot x $
it suffices to show that $G=I(\phi)$ for some $\phi \in L^p_{\bF}(\Om;\nu_p(\R_+\ti J;\R))$.
By the scalar Clark-Ocone formula of \cite[Theorem 2.1]{LaPe11b}
every random variable in $L^2(\Om,\cG)$ can be
represented as the stochastic integral of a predictable process
$\phi \in L^2(\Om;L^2(\R_+\ti J)) = L^2(\Om;\nu_2(\R_+\ti J;\R))$.
Since $G\in L^p(\Om,\cG)$, Corollary~\ref{cor:Hilbert} implies that $\phi \in L^p_{\bF}(\Om;\nu_p(\R_+\ti J;\R))$.

Finally, the uniqueness of $\cP_{{{\mathbb{F}}}}(DF)$
follows from the injectivity of $I$ as a bounded linear operator
from $L_{\mathbb{F}}^p(\Om;\nu_p(\R_+ \times J;X))$ to
$L^p(\Om;X)$.
\end{proof}

\bibliographystyle{plain}
\bibliography{poisson}
 \end{document}